\documentclass{amsart}
\usepackage[margin=1in]{geometry}

\usepackage[english]{babel}
\usepackage{amsmath}
\usepackage{amssymb}
\usepackage{amsthm}
\usepackage{latexsym}
\usepackage{tikz-cd}
\usepackage{mathtools}
\usepackage{amsfonts}
\usepackage{mathrsfs}
\usepackage{textcomp}
\usepackage[T1]{fontenc}
\usepackage{graphicx}
\usepackage{setspace}
\usepackage{nicefrac}
\usepackage{indentfirst}
\usepackage{dcolumn}
\newcolumntype{d}[1]{D{.}{.}{#1}}
\usepackage{enumerate}
\usepackage{wasysym}
\usepackage{upgreek}
\usepackage{makecell}
\usepackage{hyperref}
\usepackage{booktabs}
\usepackage{etoolbox}
\usepackage{wrapfig}
\usepackage{floatflt}
\usepackage{parskip}
\usepackage{lscape}
\usepackage[linesnumbered,commentsnumbered,ruled,vlined]{algorithm2e}
\usetikzlibrary{automata}
\newtheorem{theorem}{Theorem}[section]

\newtheorem{proposition}[theorem]{Proposition}
\theoremstyle{definition}

\newtheorem{remark}[theorem]{Remark}
\newcommand{\mydef}[1]{{\color{blue} #1}}

\newcommand{\trop}{\textrm{trop}}

\usepackage{graphicx}

\newcommand{\C}{\mathbb{C}}
\newcommand{\R}{\mathbb{R}}

\newcommand{\Z}{\mathbb{Z}}

\newtheorem{example}[theorem]{Example}
\newcommand{\bR}{\mathbb{R}}
\newcommand{\bZ}{\mathbb{Z}}
\newcommand{\bQ}{\mathbb{Q}}

\makeatletter
\newcommand{\ostar}{\mathbin{\mathpalette\make@circled\star}}
\newcommand{\make@circled}[2]{%
  \ooalign{$\m@th#1\smallbigcirc{#1}$\cr\hidewidth$\m@th#1#2$\hidewidth\cr}%
}
\newcommand{\smallbigcirc}[1]{%
  \vcenter{\hbox{\scalebox{0.77778}{$\m@th#1\bigcirc$}}}%
}
\makeatother

\title{Max-convolution through numerics and tropical geometry}

\author{Taylor Brysiewicz}
\address{Western University, Department of Mathematics, London, ON, N6A 5B7
Canada}
\email{tbrysiew@uwo.ca}
\author{Jonathan D. Hauenstein}
\address{University of Notre Dame, Department of Applied and Computational Mathematics and Statistics, Notre Dame, IN 46556 USA}
\email{hauenstein@nd.edu}
\author{Caroline Hills}
\address{University of Notre Dame, Department of Applied and Computational Mathematics and Statistics, Notre Dame, IN 46556 USA}
\email{chills1@nd.edu}

\begin{document}

\begin{abstract}
The maximum function, on vectors of real numbers, is not differentiable. Consequently, several differentiable approximations of this function 
are popular substitutes.
We survey three smooth functions which approximate the maximum function and analyze their convergence rates. We interpret these functions through the lens of tropical geometry, where their performance differences are geometrically salient. As an application, we provide an algorithm which computes the max-convolution of two integer vectors in quasi-linear time. We show this algorithm's power in computing adjacent sums within a vector as well as computing service curves in a network analysis application.
\end{abstract}

\maketitle

\section{Introduction}
Given $v=(v_1,\dots,v_n)\in\bR^n$,
although computing
the maximum $\mydef{M} = \max_{1\leq i\leq n} v_i$ is an elementary task, the function $v \mapsto \max(v)$ is not differentiable.
A common technique
used in optimization~\cite{SmoothOpt,SmoothOpt2}
and machine learning~\cite{Softmax1,Softmax2}
is to replace
the precise computation of $M$ with an approximate computation.  This article investigates three standard ways to smoothly approximate
the maximum function.
Equipped with
$$\mydef{F_v(t)} = \sum_{j=1}^n t^{v_j},
~~\mydef{L_v(t)} = \log_t(F_v(t)),~~\mydef{R_v(t)} = \dfrac{tF_v'(t)}{F_v(t)},
~~\hbox{and}~~
\mydef{||v||_p} = \left(\sum_{j=1}^n |v_j|^p\right)^{\frac{1}{p}},$$
we consider the approximations:
\[
\mydef{\text{(LogSumExp)}}:  \quad \quad M = \displaystyle\lim_{t \to \infty} L_v(t) \quad \quad \quad \quad \quad \quad \quad 
\mydef{\text{(Ratio)}}:  \quad \quad M = \displaystyle\lim_{t \to \infty} R_v(t)
\]
and, if each entry of $v$ is non-negative,
\[
\mydef{\text{($p$-norm)}}: \quad \quad M= \displaystyle\lim_{p \to \infty} ||v||_p = ||v||_\infty.
\]
We drop the subscript $v$ when the vector of interest is clear from context.

\begin{figure}[!htpb]
\includegraphics[scale=0.063]{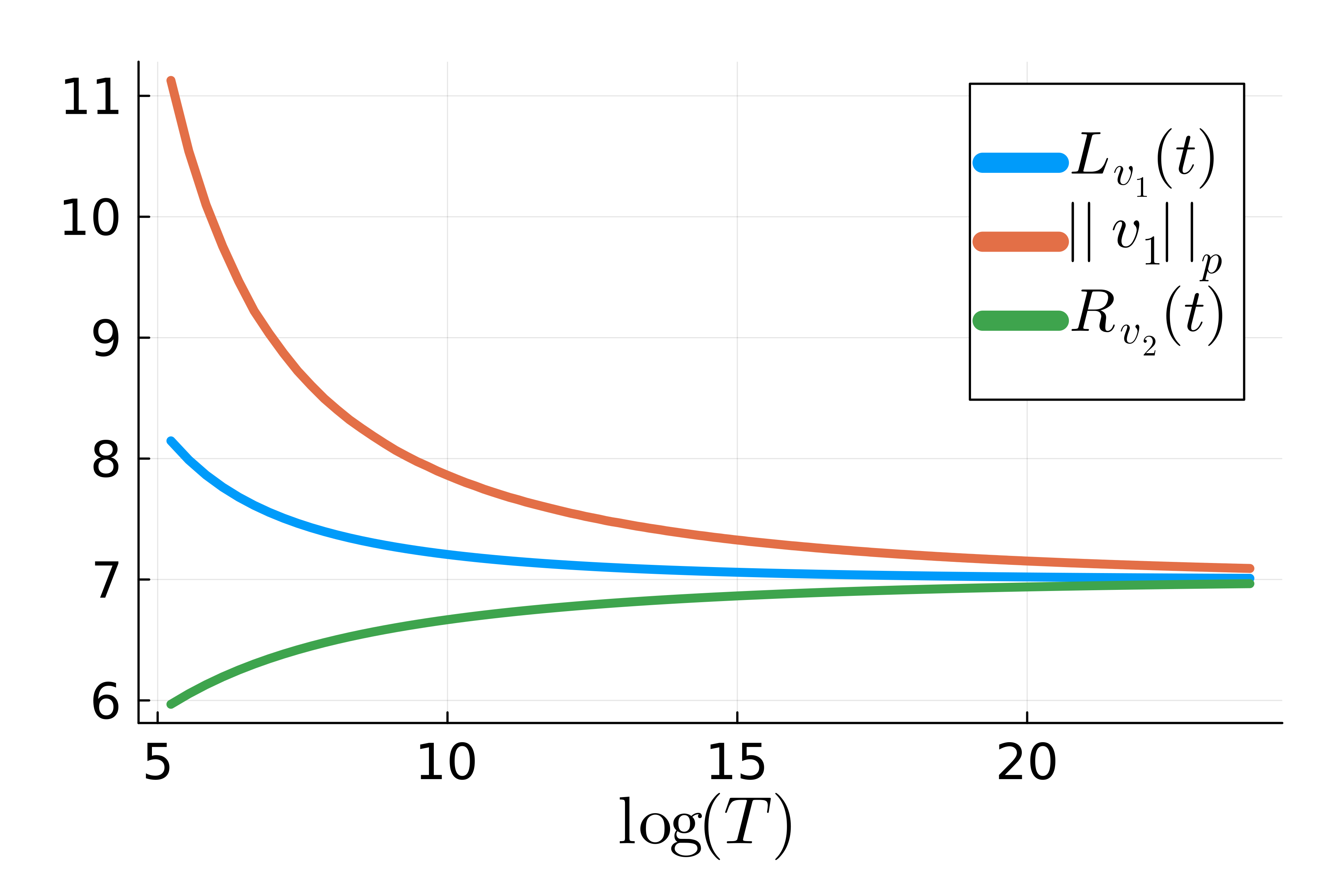}
\includegraphics[scale=0.063]{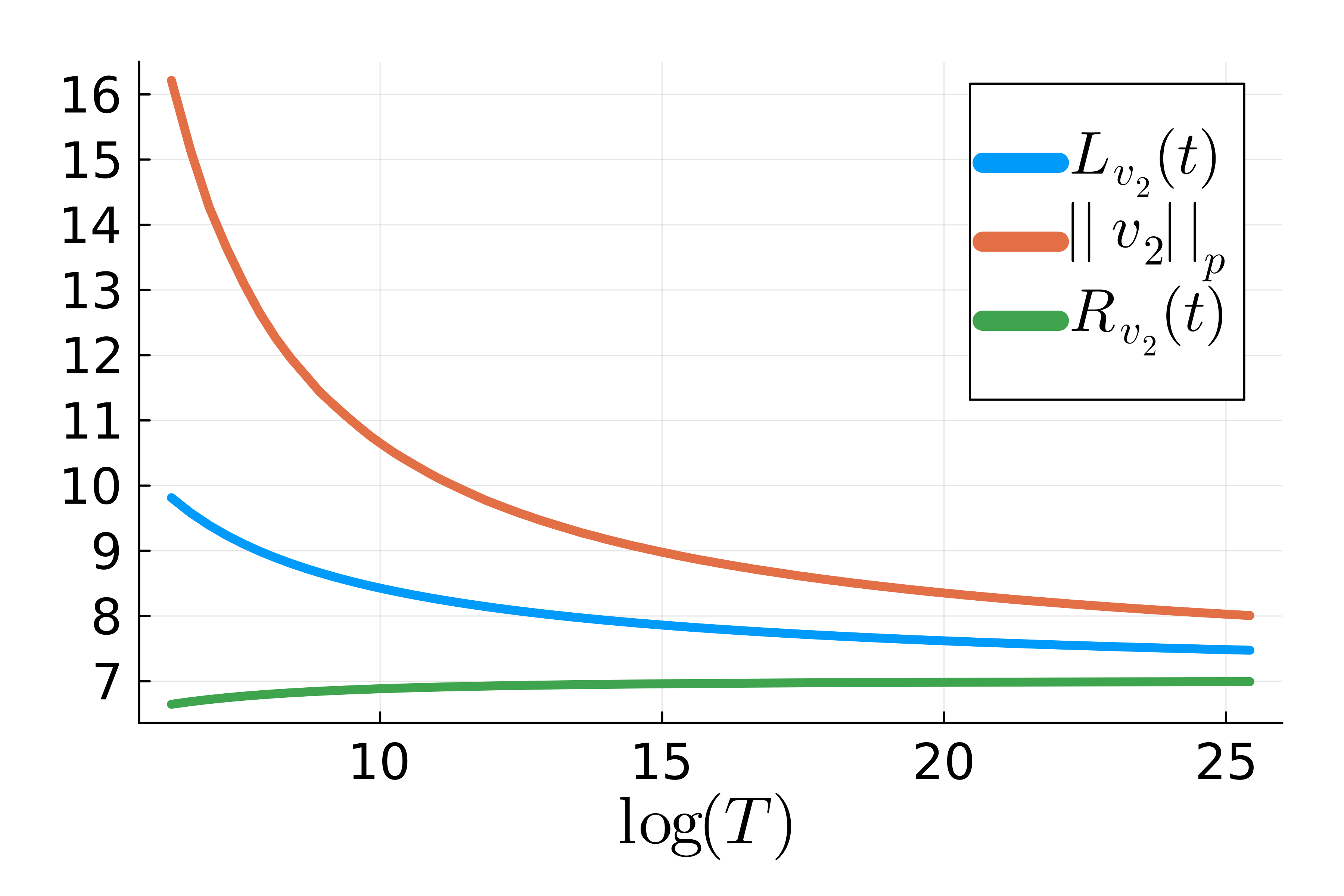}
\caption{Plots of the values of the smooth approximations $L(t),||v||_{p}$, and $R(t)$ for the vectors $v_1=(1,2,3,4,5,6,7)$ and $v_2=(1,2,3,4,5,6,7,7,7,7,7,7)$. These values are plotted against the natural logarithm of the largest absolute value $T$ of a floating point number involved in the numerical evalutation of each function.}
\end{figure}

When $t>0$, the functions
$L_v(t)$ and $R_v(t)$ are smooth
as a function of $v$
and approximate $M$
as shown by the limits above.
For $p \in \mathbb{R}_{>0}$, the function $||v||_p$ smoothly approximates $M$ provided that each element of $v$ is non-negative.
Each of these functions can be
expressed in terms of
\begin{equation}\label{eq:Lv}
\mydef{\mathcal L_v(t)} = \log(F_v(t)).
\end{equation}

\begin{proposition} \label{prop:Lexpressions}
For $v=(v_1,\ldots,v_n) \in \bR^n$, $t \in \bR$
with $t> 1$, we have $\mydef{u} = \log(t)>0$, and
\begin{align}
L_v(t) = L_v(e^u) &= \frac{1}{u}\mathcal L_v(e^u) \\
\label{eq:Rv}
R_v(t) = R_v(e^u) &= \frac{d}{d u} \mathcal L_v(e^u) \\
||v||_{\log(t)} = ||v||_{u} &= e^{L_{\log|v|}(t)}=e^{\frac{1}{u} \mathcal L_{\log|v|}(e^u)}
\end{align}
where $\mydef{\log |\cdot|}:\R^n \to \R^n$ is the componentwise log-absolute value map.
\end{proposition}

After providing some basic notation
in Section~\ref{sec:Basic}, we derive convergence rates for each of these functions in
Section~\ref{sec:Approximations} by analyzing~$\mathcal L_v(t)$. Namely, for $\delta>0$, we give bounds on $t$ for which the absolute error of these approximations is smaller than $\delta$. As a consequence, when the vector $v$ is integral, one may use a rounding procedure, such as the floor or ceiling function, to provably compute $M$ via a single evaluation of an approximating function.

In practice, when the entries of $v$ are selected
from a discrete set, such as the integers~$\bZ$,
the maximum $M$ is often obtained
more than once.  The number of times
$M$ is attained is
the \emph{multiplicity} of $M$ in $v$,
namely
$$\mydef{\mu_M} = \# \{i~|~v_i = M\}.$$ When the multiplicity of $M$ is large, the ratio approximation significantly outperforms the other approximations (see Section~\ref{sec:Experiments}).
In Section~\ref{sec:Approximations}, we express $\mu_M$ as a limit of the aforementioned functions and derive analogous convergence rates
result for computing $\mu_M$.

In light of part \eqref{eq:Rv} of Proposition~\ref{prop:Lexpressions},
we generalize the approximation $R_v(t)$
using higher-order derivatives.
In particular, for $k\geq 1$, we define
$$
\mydef{R_v^{(k)}(t)} = -\frac{(-t)^k}{(k-1)!} \frac{d^k}{dt^k} \mathcal L_v(t).
$$
Observe that $R_v^{(1)}(t) = R_v^{(1)}(e^u) = \frac{d}{du} \mathcal L_v(e^u) = R_v(e^u) = R_v(t)$.
We show, in Section~\ref{sec:Approximations},
that every $R_v^{(k)}(t)$ for $k\geq 1$ converges
to $M$ at the same rate.
We discuss how to use these higher-order derivatives to numerically approximate other information about $v$ (see Theorem \ref{thm:HigherRatioOtherInformation}).

In Section \ref{sec:Tropical}, we explore
the geometry of $L_v(t)$ and $R^{(k)}_v(t)$ in terms of objects called amoebas from the world of tropical geometry. We realize the graph of the function $u \mapsto \mathcal L_v(e^u)$ as the upper boundary of a certain amoeba and provide a geometric
interpretation of the performance differences of $R_v(t)$ and $L_v(t)$ when $\mu_M$ is large.

In Section \ref{sec:Experiments}, we conduct a series of experiments showcasing our theoretical results
and the performance differences of the approximation techniques discussed. In particular, we provide empirical evidence showing the extent to which the bounds derived in Section \ref{sec:Approximations} are tight. We illustrate how the ratio approximations perform significantly better than the others when the maximum appears with non-trivial multiplicity and that this feature persists in the presence of a noisy model.

In Section \ref{sec:MaxCon}, we
propose an algorithm
for the max-convolution problem:
\begin{align}
\nonumber\mydef{\textbf{MAXCON:}} \quad &\text{Given }a=(a_0,\dots,a_n)\in\bZ^{n+1} \hbox{~and~} b=(b_0,\dots,b_n) \in \bZ^{n+1}, \\
& \text{compute } c=(c_0,\dots,c_{2n})\in\bZ^{2n+1} \hbox{~where}\\
& \nonumber
c_k = \max_{\max\{0,k-n\}\leq i \leq \min\{k,n\}} (a_i+b_{k-i}).
\end{align}
Since each $c_k$ is a maximum of the integer vector $v^{(k)}=(a_i+b_{k-i})_{i=\max\{0,k-n\}}^{\min\{k,n\}}$ its value may be determined by an (appropriately large) evaluation of an approximation of that maximum. In particular, any algorithm which computes classical convolution coefficients may be used as an oracle for evaluating $L_{v^{(k)}}(t)$. The fast Fourier transform, for example, performs such a computation using $O(n\log(n))$ operations. By combining this fact with our bounds from Section \ref{sec:Approximations}, we obtain a
quasi-linear time algorithm for the max-convolution problem. We end by applying our numerical approach to the {maximum consecutive subsums problem} and and the computation of {service curve constraints}.

\section{Notation and Fundamental Results}\label{sec:Basic}
We begin by fixing the following notation
\begin{align*}
\mydef{v=(v_1,\ldots,v_n):}& \text{ an }n\text{-tuple of real numbers} \\
\mydef{M:}& \,\, \max(v) \\
\mydef{\mu_c:}& \text{ multiplicity of a real number }c\text{ in }v, \text{ i.e., } \#\{i~|~v_i = c\} \\
\mydef{\ell:} & \text{ number of distinct elements in }v \\
\mydef{w=(w_1,\ldots,w_\ell):}& \text{ decreasing list of unique elements in }v\text{ i.e., } M=w_1>\cdots>w_\ell \\
\mydef{g=(g_1,\ldots,g_\ell):}& \, \, g_i = M-w_i
\text{ with } 0 = g_1 < g_2 < \cdots < g_\ell.
\end{align*}
Additionally, $\mydef{t}$ will denote a variable which takes on positive real values whereas $\mydef{u}=\log(t)$ is its image under the natural logarithm.

\begin{example}
To illustrate notation,
consider $v= (7,7,-1,0,1,1,2.5,2.5,7,7)~\in~\bR^{10}$.
$$ M=7, \quad \ell = 5,$$
$$\mu_7 = 4, \quad \mu_{2.5}=2, \quad \mu_1=2, \quad \mu_0=1, \quad \mu_{-1}=1,$$
$$w=(7,2.5,1,0,-1), \quad g=(0,4.5,6,7,8).$$
\end{example}

In principle,
the elements of $v$ can be any real numbers.
However, in practice, they
are usually some rational floating point approximations, that is, $v\in\bQ^n$. Moreover, we may assume for our analyses that $v \in \mathbb{Z}^n$ since evaluating $\mathcal L_v(t)$ at a power $t^k$ of $t$ corresponds to scaling $v$ by $k$:
\begin{equation}
\label{eq:powerscaling}
\mathcal L_v(t^k)=\mathcal L_v(e^{ku})=\mathcal L_{kv}(e^u)=\mathcal L_{kv}(t).
\end{equation}

\begin{remark}
We note that $\mathcal L_v(t) = \mathcal L_{uv}(e)$.  The function $\mydef{\textrm{lse}(v)} = \mathcal L_v(e)$ is a popular activation function in the  field of machine learning traditionally called the \emph{log-sum-exp} function \cite{NielsenSun}.
Numerical methods for accurately evaluating
it may be found in \cite{BHH}.
 \end{remark}

The following expansion of $\mathcal L_v(t)$ near $t=\infty$ is fundamental for our analysis.

\begin{proposition}\label{prop:LLexpansion}

For $v \in \Z^n$, $\mathcal L_v(t)$
near $t=\infty$ has the following expansion:
\begin{align}
\label{eq:LLexpansion}
\mathcal L_v(t) & = \log(t)M+\log(\mu_M) + \log\left(\sum_{i=1}^\ell \frac{\mu_{w_i}}{\mu_M} t^{-g_i} \right)
\end{align}
where the first term inside the logarithm is
$\frac{\mu_{w_1}}{\mu_M}t^{-g_1} = 1$.
In particular, by expanding the logarithmic term, there exists nonnegative real numbers
$\{\alpha_j\}_{j=1}^\infty$ such that
\begin{align}
\label{eq:LLexpansion2}
\mathcal L_v(t)&=\log(t)M+\log(\mu_M) + \sum_{j=1}^\infty \alpha_j t^{-j} \\
&= uM+\log(\mu_M)+\sum_{j=1}^\infty \alpha_je^{-ju}
\end{align}
where $\alpha_1 = \cdots = \alpha_{g_2-1} = 0$ and $\alpha_{g_2} = \dfrac{\mu_{w_2}}{\mu_M}$.
\end{proposition}

We obtain similar expressions for $L_v(t)$, $R_v(t)$, and $R_v^{(k)}(t)$ by combining Propositions~\ref{prop:Lexpressions} and~\ref{prop:LLexpansion}.
\begin{proposition}
\label{prop:Expansions}
For $v \in \Z^n$,
let $\{\alpha_j\}_{j=1}^\infty$ be as
in Proposition~\ref{prop:LLexpansion}.
Then, near $t=\infty$:
\begin{align}
\label{eq:Lexpansions}
L_v(t) &= M + \log_t(\mu_M) + \log_t\left(\sum_{i=1}^\ell \frac{\mu_{w_i}}{\mu_M} t^{-g_i} \right) \\ &= M + \frac{1}{\log(t)}\left(\log(\mu_M) + \sum_{j=1}^\infty \alpha_jt^{-j}\right) \notag \\
&= M+\frac{1}{u}\left(\log(\mu_M) + \sum_{j=1}^\infty \alpha_j e^{-ju}\right),\notag \\
\label{eq:RExpansion}
R_v(t)&=  M-\sum_{j=1}^\infty j\alpha_j t^{-j} \\
\notag &=  M-\sum_{j=1}^\infty j\alpha_je^{-ju}, \\
\label{eq:RkExpansions}
R_v^{(k)}(t)&=M-\sum_{j=1}^\infty j \binom{j+k-1}{k-1} \alpha_j t^{-j} \\
&=M-\sum_{j=1}^\infty j \binom{j+k-1}{k-1} \alpha_je^{-ju}.\notag
\end{align}
In particular, $R_v(t)$ and, more generally, $R_v^{(k)}(t)$ are analytic at $t=\infty$.
Additionally, if each $v_i\geq 0$,
\begin{align}
\label{eq:Pexpansion}
||v||_{\log(t)} &= e^{L_{\log|v|}(t)} = M \cdot e^{L_{\log|v|}(t)-\log(M)}.
\end{align}
\end{proposition}
Clearly, $L_v(t)\geq M$ and the $\log_t(\mu_M)$ term
from \eqref{eq:Lexpansions} is responsible for a slow convergence rate of $L_v(t)$ to~$M$.
This logarithmic term is eliminated in $R_v(t)$
and, more generally, in  $R_v^{(k)}(t)$. See
Section~\ref{sec:Tropical} for a geometric explanation of this fact.

\begin{remark}
\label{rem:Cauchy}
Since $R_v(t)$ and, more generally,
$R_v^{(k)}(t)$ are analytic at $t=\infty$
when $v\in\bZ^n$,
Cauchy's integral formula yields that
for each $k\geq 1$ there exists $r > 0$ such that
\begin{equation}
\label{eq:CauchyIntegral}
\frac{1}{2\pi \sqrt{-1}} \oint_{|t|=r} t^{-1}\cdot R_v^{(k)}(t^{-1})\cdot dt = M.
\end{equation}
Numerically, one can use the trapezoid rule~\cite{TrapezoidRule}
to approximate $M$ from this integral.
\end{remark}

Since $R_v(t)$ depends on $F_v'(t)$ which may be difficult to evaluate in practice (see Section~\ref{sec:MaxCon}), we show below how to approximate
$R_v(t)$ from evaluations of $L_v(t)$.

\begin{proposition}\label{prop:ApproxR}
For $v\in\bR^n$, $t>1$, and $\alpha>0$ with $\alpha\neq 1$,
define
\begin{equation}\label{eq:DifferenceApprox}
\mydef{D_v(t,\alpha)} = \log_\alpha\left(\frac{F_v(\alpha\cdot t)}{F_v(t)}\right)
= \frac{\mathcal L_v(\alpha\cdot t) - \mathcal L_v(t)}{\log(\alpha)}.
\end{equation}
Then,
$$\lim_{\alpha\rightarrow 1} D_v(t,\alpha) = R_v(t).$$
\end{proposition}
\begin{proof}
Applying l'H\^opital's rule yields $\lim_{\alpha\rightarrow 1} D_v(t,\alpha) = \lim_{\alpha\rightarrow 1} R_v(\alpha\cdot t) = R_v(t).$
\end{proof}

\section{Approximating quantities associated to \texorpdfstring{$v$}{v}}
\label{sec:Approximations}
Equipped with the expansions \eqref{eq:Lexpansions}--\eqref{eq:Pexpansion} and \eqref{eq:DifferenceApprox},
the functions $L_v(t)$,  $R_v^{(k)}(t)$,
$D_v(t,\alpha)$, and $||v||_u$ may each be used to approximate certain information about $v$, such as $M$, $\mu_M$, and $g_2$.

For each such approximation
of $M$, we derive a lower bound
on~$t$ so that the absolute error is less
than a given value $\delta>0$.
For integer vectors, we pay particular attention to the case where $\delta=1$, since one
can use the floor $\mydef{\lfloor \cdot \rfloor}$ and ceiling~$\mydef{\lceil \cdot \rceil}$ functions to \emph{provably} compute these values from their approximations.

\subsection{Computing the maximum}
\label{secsec:ComputingMaximum}
We derive bounds on the absolute errors of $L_v(t)$, $R_v(t)$, $D_v(t,\alpha)$, and~$||v||_u$ in Theorems \ref{thm:Ltbounds}, \ref{thm:Rtbounds}, \ref{thm:Dtbounds}, and \ref{thm:MaxFromPnormBound} respectively.
\begin{theorem}
\label{thm:Ltbounds}
Fix $v \in \mathbb Q^n$ and $\delta>0$. Then
$0\leq L_v(t)-M<\delta$ whenever $t>1$ and
$$t^{\delta+g_2}-t^{g_2}\mu_M-(n-\mu_M)>0.$$
If $v \in \Z^n$ and $\delta=1$, this bound is obtained when
$$t > \frac{\mu_M+\sqrt{\mu_M^2+4(n-\mu_M)}}{2}.$$
If additionally $\mu_M=1$ then this bound simplifies to $t>\dfrac 1 2 + \sqrt{n}$.
\end{theorem}
\begin{proof}
Assume, after reindexing, that $v_1=\cdots=v_{\mu_M}=M$. Thus,
$$L_v(t) - M = \log_t\left(\mu_M + \sum_{j=\mu_M+1}^{n} t^{v_j-M}\right).$$
Hence, $L_v(t)-M<\delta$ provided that the expression within the logarithm is smaller than $t^{\delta}$. Since the function $t^x$ is monotonic for $t>1$,
$$\mu_M + \sum_{j=\mu_M+1}^{n} t^{v_j-M} \leq \mu_M + (n-\mu_M)t^{-g_2},$$
completing the proof of the first statement
since this value is less than $t^{\delta}$ when
$$t^{\delta} > \mu_M + (n-\mu_M)t^{-g_2}.$$
For $v\in\bZ^n$, we have that $g_2\geq 1$ so that
a sufficient condition
when $\delta =1$ is
$$t^2-\mu_M t - (n-\mu_M)>0$$
yielding the second statement.
The third statement follows immediately.
\end{proof}
When $v$ consists of integers and $\mu_M$ is known, Theorem \ref{thm:Ltbounds} suggests an algorithm which provably computes~$M$ using one evaluation of $L_v(t)$:
\begin{center}
\emph{Return $\lfloor L_v(t) \rfloor$ for $t$ satisfying the inequality $2t>\mu_M+\sqrt{\mu_M^2+4(n-\mu_M)}$. }
\end{center}
The largest $t$ value required is when $\mu_M = n$
for which one can take $t = n+1$.
In particular, for any $v\in\bZ^n$,
one always has $\lfloor L_v(n+1) \rfloor = M$.

\begin{figure}[b!]
\includegraphics[scale=0.065]{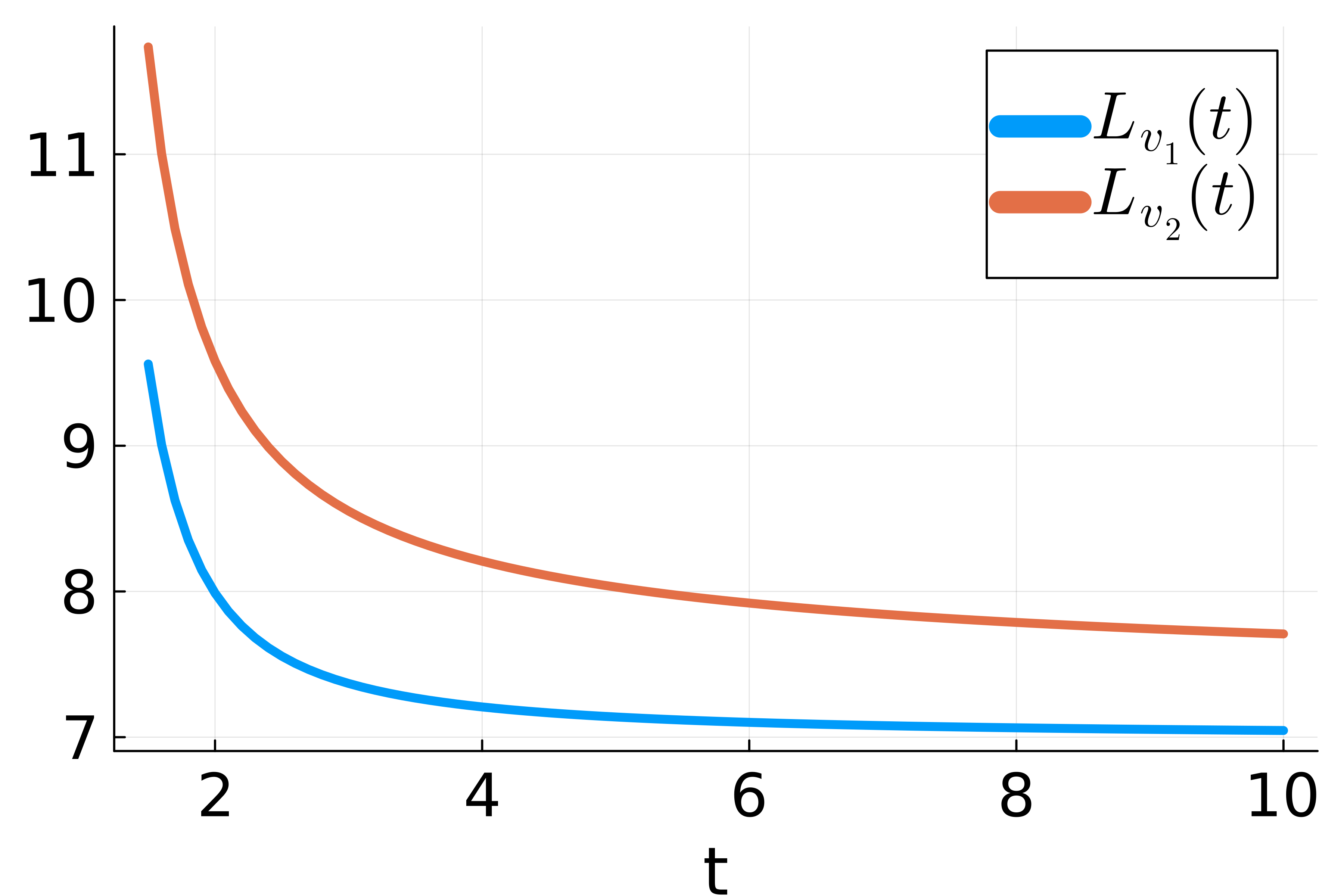}
\caption{The graphs of $L_{v_1}(t)$ and $L_{v_2}(t)$ as in Example~\ref{ex:simple_integer_multiplicity_difference}.}
\label{fig:simple_integer_multiplicity_difference}
\end{figure}
The following example illustrates Theorem \ref{thm:Ltbounds} on qualitatively different input.
\begin{example}
\label{ex:simple_integer_multiplicity_difference}
Consider the following integer vectors:
\[v_1 = (1,2,3,4,5,6,7)\hspace{0.25 in} \text{and} \hspace{0.25 in} v_2=(1,2,3,4,5,6,7,7,7,7,7).
\]
The maximum of both vectors is $7$
which has multiplicity $1$ and $5$ in
$v_1$ and $v_2$, respectively.
By Theorem~\ref{thm:Ltbounds},
$L_{v_1}(t) \in [7,8)$ when $t>3$ and $L_{v_2}(t) \in [7,8)$ when $t>6$. Figure \ref{fig:simple_integer_multiplicity_difference} displays a verification of these bounds and illustrates the reduced convergence rate for $v_2$ due to the increased multiplicity of the maximum.
\end{example}

The worst-case scenario analysis for $R_v(t)$ is qualitatively distinct from that of~$L_v(t)$. The fact which distinguishes these cases is that
for a fixed $t>1$, the function $x \mapsto xt^{-x}$ is decreasing only after reaching its maximum on $\mathbb{R}_{>0}$ at $x=\log(t)^{-1}$.

\begin{theorem}
\label{thm:Rtbounds}
Fix $v \in \mathbb Q^n$ and $\delta>0$. Then $0 \leq M - R_v(t) <\delta$ {when $t>e^{1/g_2}$ and} \[
t > \left(\frac{(n-\mu_M)g_2}{\delta \cdot \mu_M}\right)^{\frac{1}{g_2}}.
\]
If $v \in \Z^n$ and $\delta=1$, this bound is obtained when
\[
t > \max\left(e,\frac{n-\mu_M}{\mu_M} \right).
\]
\end{theorem}
\begin{proof}
From~\eqref{eq:RExpansion}, a worst-case analysis
with $t>e^{\frac{1}{g_2}}$ shows that
$$M-R_v(t) =  \sum_{j=1}^\infty j \alpha_j t^{-j} < \frac{n-\mu_M}{\mu_M} g_2 t^{-g_2}.$$
Therefore, the main result follows from
\[M-R_v(t)<\delta \quad \quad \text{whenever} \quad \quad t^{g_2}>\frac{(n-\mu_M)g_2}{\delta \cdot \mu_M}.\]
When $v \in \mathbb{Z}^n$ and $\delta=1$, this simplifies to $t^{g_2}>\frac{(n-\mu_M)g_2}{\mu_M}$.  Since $t>e$ and $g_2\geq 1$, this holds
if additionally
$$t>\frac{n-\mu_M}{\mu_M}.$$
\end{proof}

\begin{example}
For $v_1$ and $v_2$ as in Example~\ref{ex:simple_integer_multiplicity_difference},
Figure \ref{fig:multiplicityDifference}
compares the graphs of $L(t)$ and $R(t)$.  Note that Theorems \ref{thm:Ltbounds} and \ref{thm:Rtbounds} guarantee that $L_{v_2}(t) \in[7,8)$ when $t>6$ and $R_{v_2}(t)\in (6,7]$ when $t>e\approx 2.718$.

\end{example}

\begin{figure}[htpb!]
\includegraphics[scale=0.065]{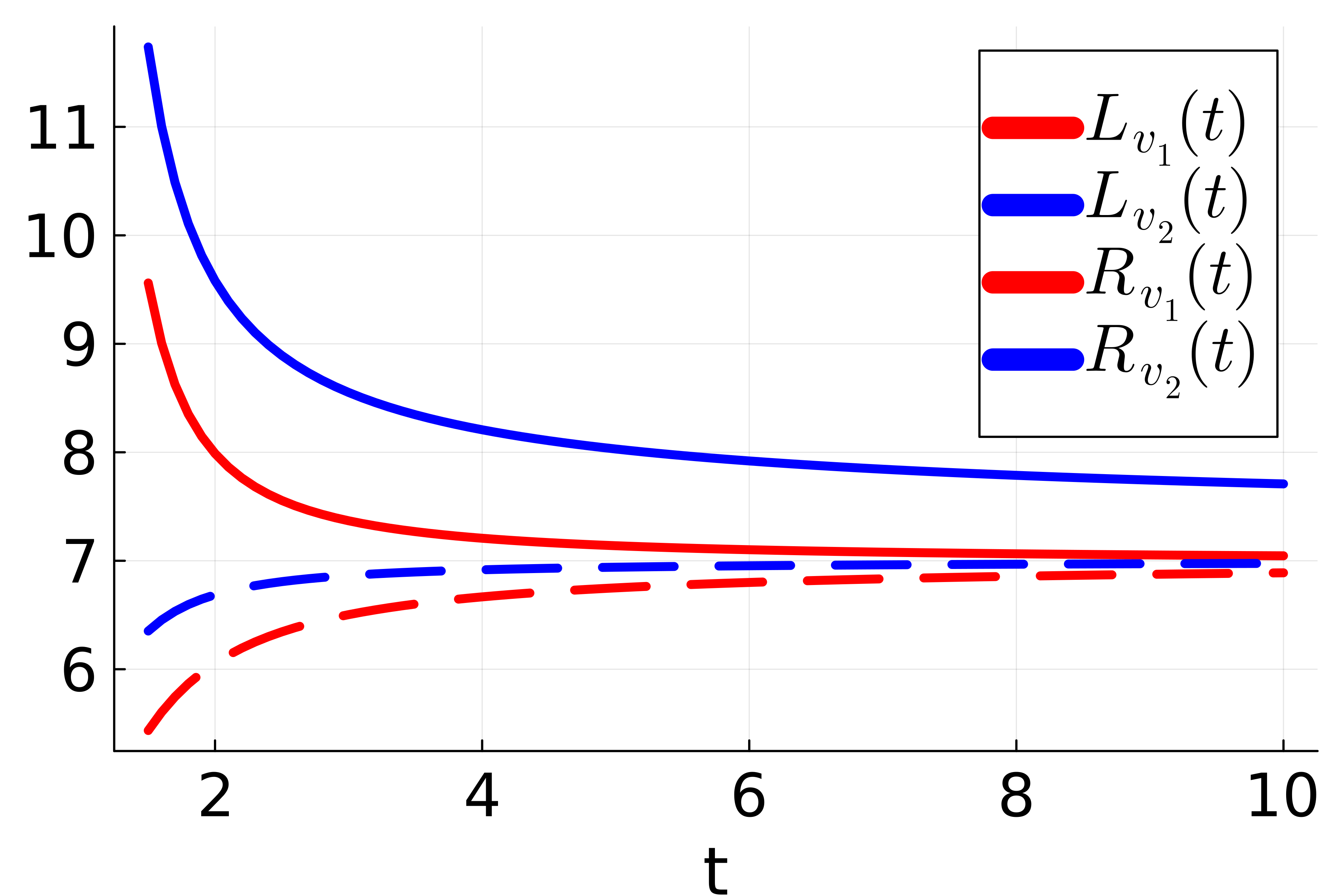}
\caption{The graphs of $L(t)$ and $R(t)$ applied to $v_1$ and $v_2$ from  Example~\ref{ex:simple_integer_multiplicity_difference}.
}
\label{fig:multiplicityDifference}
\end{figure}

\begin{remark}
For $v\in\bZ^n$, one has
$$L_v(t)-M = \left\{\begin{array}{ll}
O(1/\log(t)) & \hbox{~~if~} \mu_M > 1, \\
O(t^{-g_2}/\log(t)) & \hbox{~~if~} \mu_M = 1,
\end{array}\right.
\hbox{~~~~~~~~and~~~~~~~~}
M-R_v(t) = O(t^{-g_2}).
$$
When $\mu_M=1$, Thereom~\ref{thm:Ltbounds}
requires $t=O(\sqrt{n})$
while Theorem~\ref{thm:Rtbounds}
requires $t=O(n)$.
However, the bound $\frac{n-\mu_M}{\mu_M}$
from Theorem~\ref{thm:Rtbounds} is smaller than the bound $\frac{\mu_M + \sqrt{\mu_M^2 + 4(n-\mu_M)}}{2}$ from Theorem \ref{thm:Ltbounds} whenever $\mu_M \geq \frac{1}{4}(\sqrt{8n+1}-1)$. For reference, this means that
$$(n,\mu_M)\in\{(10,2),(105,7),(1081,23),(10153,71),(100576,224),\ldots\}$$ are afforded equal $t$-bounds for $L_v(t)$ or $R_v(t)$ via
Theorems~\ref{thm:Ltbounds} and~\ref{thm:Rtbounds}, respectively. These bounds are derived from the worst-case scenarios where $M-1$ appears with multiplicity $n-\mu_M$. However, on input vectors~$v$ sampled from the uniform distribution on $\{0,\ldots,M\}$ with varying multiplicities $\mu_M$, $L_v(t)$ consistently  performs worse than $R_v(t)$. For more details, see the experiments in Section \ref{sec:Experiments}.
\end{remark}

Based on the relationship between $D_v(t,\alpha)$
and $R_v(t)$ summarized in Proposition~\ref{prop:ApproxR},
the error is similar to Theorem~\ref{thm:Rtbounds}.  Here, the worst case analysis
yields the function $x\mapsto \dfrac{t^{-x}(1-\alpha^{-x})}{\log(\alpha)}$ which is decreasing after reaching its maximum on~$\bR_{>0}$ at $x = \dfrac{\log(\log(\alpha t))-\log(\log(t))}{\log(\alpha)}$ which limits to $\log(t)^{-1}$ as $\alpha\rightarrow1$.

\begin{theorem}\label{thm:Dtbounds}
Fix $v\in\bQ^n$, $\delta > 0$, and $\alpha>1$. Then, $0 \leq M - D_v(t,\alpha) < \delta$
when $t > e^{1/g_2}$,
$$
\begin{array}{l}
t > \alpha^{\dfrac{\alpha^{-g_2}}{1-\alpha^{-g_2}}}
\xrightarrow{\alpha\rightarrow1} e^{1/g_2},
\hbox{~~and~~} \\[0.1in]
t > \left(\dfrac{n-\mu_M}{\delta\cdot \mu_M}\cdot \dfrac{1-\alpha^{-g_2}}{\log(\alpha)}\right)^{\frac{1}{g_2}}\xrightarrow{\alpha\rightarrow1} \left(\dfrac{(n-\mu_M)g_2}{\delta\cdot \mu_M}\right)^{\frac{1}{g_2}}.
\end{array}$$
If $v\in\bZ^n$ and $\delta = 1$, this bound is obtained when $\alpha > 1$ and
$$t > \max\left(e,\frac{n-\mu_M}{\mu_M}\right).
$$
\end{theorem}
\begin{proof}
The worst case analysis using the three assumptions
on $\alpha$ and $t$ that are independent of $\delta$
show that
\[M-D_v(t,\alpha)<\delta \quad \quad \text{whenever} \quad \quad t^{g_2}>\dfrac{n-\mu_M}{\delta\cdot \mu_M}\dfrac{1-\alpha^{-g_2}}{\log(\alpha)}.\]
When $v \in \mathbb{Z}^n$ and $\delta=1$, this simplifies to $t^{g_2}>\frac{n-\mu_M}{\mu_M}\frac{1-\alpha^{-g_2}}{\log(\alpha)}$.  Since $t>e$, $\alpha>1$, and $g_2\geq 1$, this holds
if additionally $t>\frac{n-\mu_M}{\mu_M}$.
\end{proof}

To analyze the $p$-norm case, following \cite{PfeufferSerang2016,Serang2015},
we assume the vector $v$ has undergone a linear transformation so that each $v_j \in [0,1]$.

\begin{theorem}
\label{thm:MaxFromPnormBound}
If $v \in [0,1]^n$ and $\delta >0$, then $0\leq ||v||_u - M<\delta$ when $u>1$ and
$$e^{u(\Delta+g_2)}-e^{ug_2}\mu_M-(n-\mu_M)>0 \quad \quad \text{where} \quad \Delta= \log\left(1+\frac{\delta}{M}\right).$$
\end{theorem}
\begin{proof}
By \eqref{eq:Pexpansion}, $||v||_{u} = M\cdot e^{\epsilon(u)}$ where $\epsilon(u)=L_{\log|v|}(e^u)-\log(M)$ which is the
error when using $L_{\log|v|}(e^u)$
to approximate $\log(M)$. Hence, $||v||_{u} - M < \delta$ if and only if $\epsilon(u) < \log \left(1+\frac{\delta}{M}\right)=:\Delta$. By Theorem \ref{thm:Ltbounds}, this occurs whenever
$t>1$ and
 $$t^{\Delta+g_2}-t^{g_2}\mu_M - (n-\mu_M) >0.$$
Since $u=\log(t)$, changing coordinates gives the result.
\end{proof}

\begin{figure}
\includegraphics[scale=0.065]{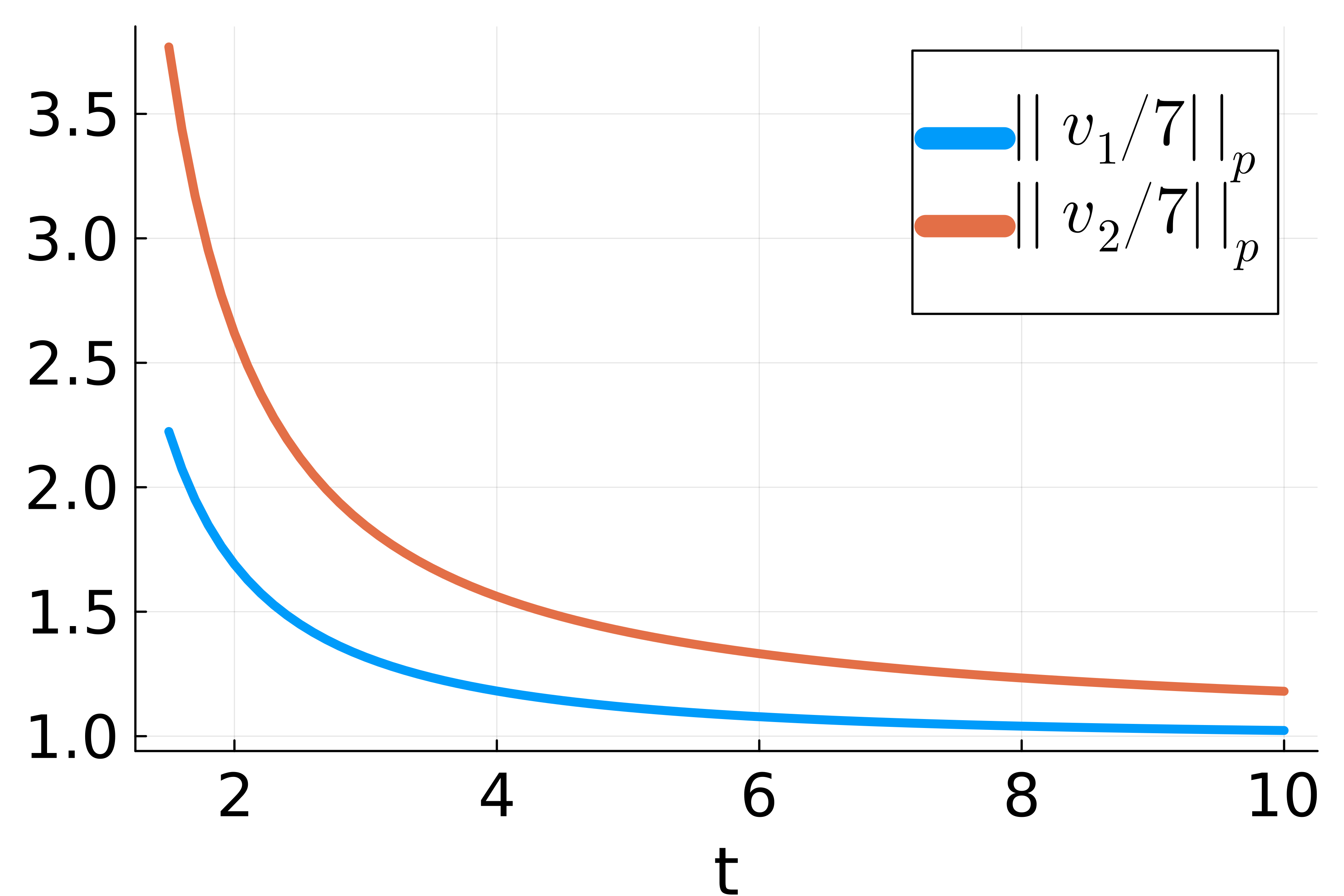}
\caption{The graphs of the $p$-norms of the vectors $v_1/7$ and $v_2/7$ with entries in $[0,1]$, where $v_1$ and $v_2$ are from  Example~\ref{ex:simple_integer_multiplicity_difference}.
}
\label{fig:pnormMultiplicityDifference}
\end{figure}

\begin{figure}[!htpb]
\begin{center}
\includegraphics[scale=0.055]{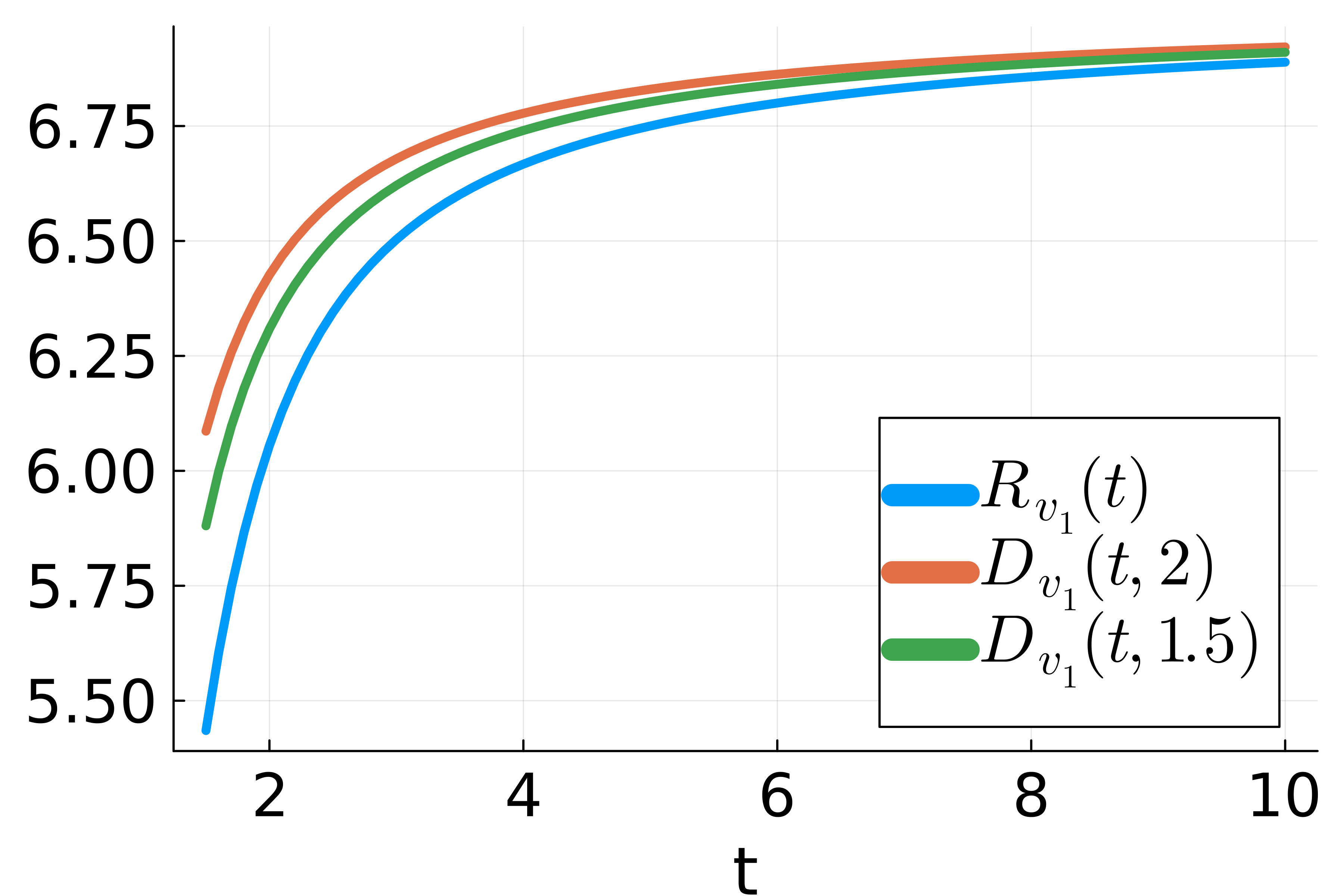} \quad \quad 
\includegraphics[scale=0.055]{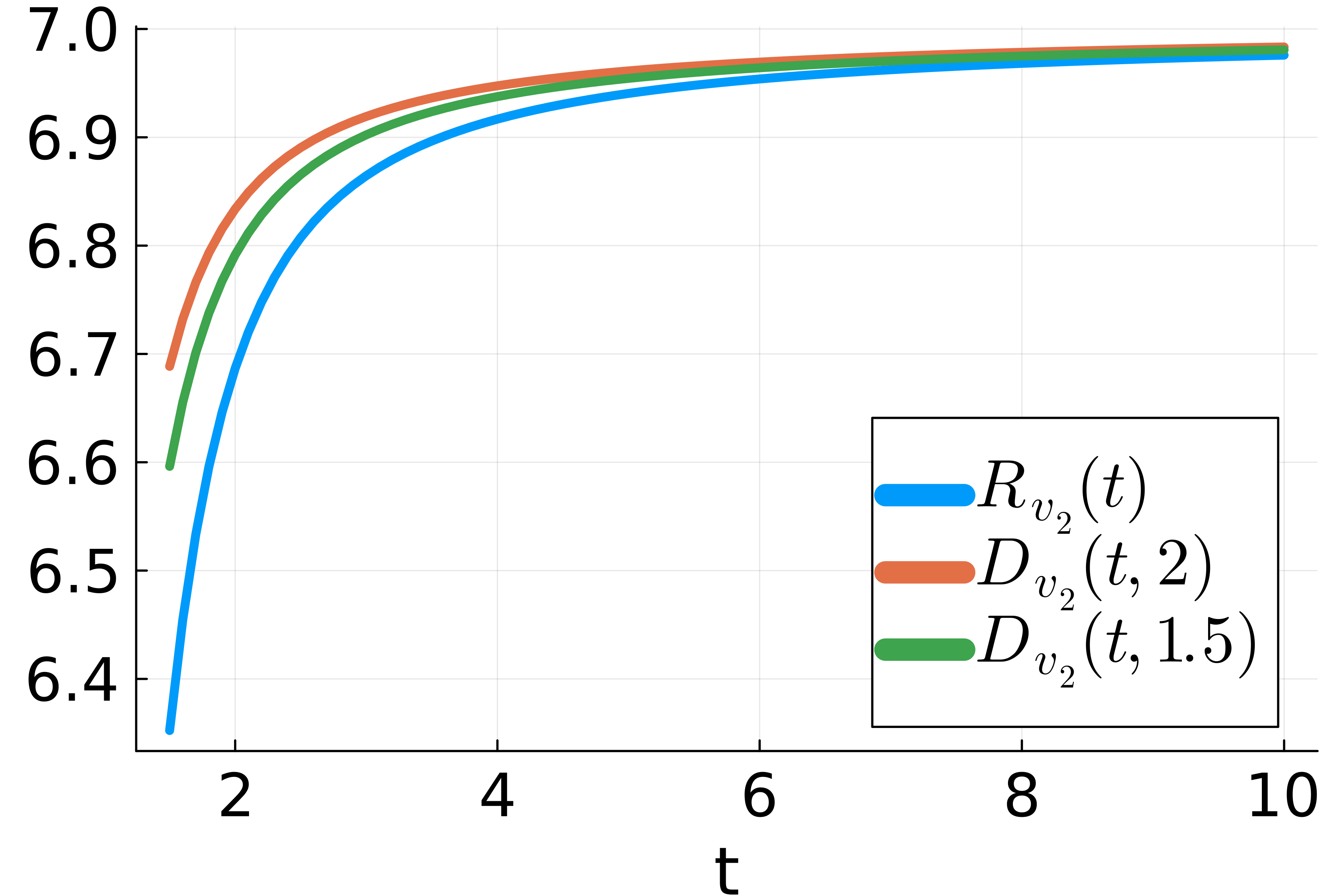}
\end{center}
\caption{Comparison of $D_{v}(t,\alpha)$ for $v_1$ and $v_2$ from Example \ref{ex:simple_integer_multiplicity_difference} with various values of $\alpha$.}
\label{fig:DplotComparison}
\end{figure}

\begin{example}
The following illustrates the differences between the approximations $L(t), R(t), D(t,\alpha),$ and~\mbox{$||\cdot ||_p$} of $M$ on our running examples of $v_1$ and $v_2$
from Example~\ref{ex:simple_integer_multiplicity_difference}.
First, similar to previous plots,
Figure~\ref{fig:pnormMultiplicityDifference}
shows the difference of convergence
rate for the $p$-norm approximation
due to higher multiplicity.
Next, Figure~\ref{fig:DplotComparison}
compares $R_v(t)$ with $D_v(t,2)$
and $D_v(t,1.5)$ for $v=v_1$ and $v=v_2$
showing comparable convergence rates.
Finally, we compare the values of
$||v||_p, L_v(t)$, $R_v(t)$, and $D_v(t,\alpha)$ when they require comparably large (in absolute value) floating point number for evaluation.
Setting $\mydef{T}$ to be the largest floating point number required, we plot these functions against $\log(T)$ in Figures~\ref{fig:Full_M_Comparison_v1} and~\ref{fig:Full_M_Comparison_v2}
for $v=v_1$ and $v=v_2$, respectively.

\begin{figure}
\includegraphics[scale=0.065]{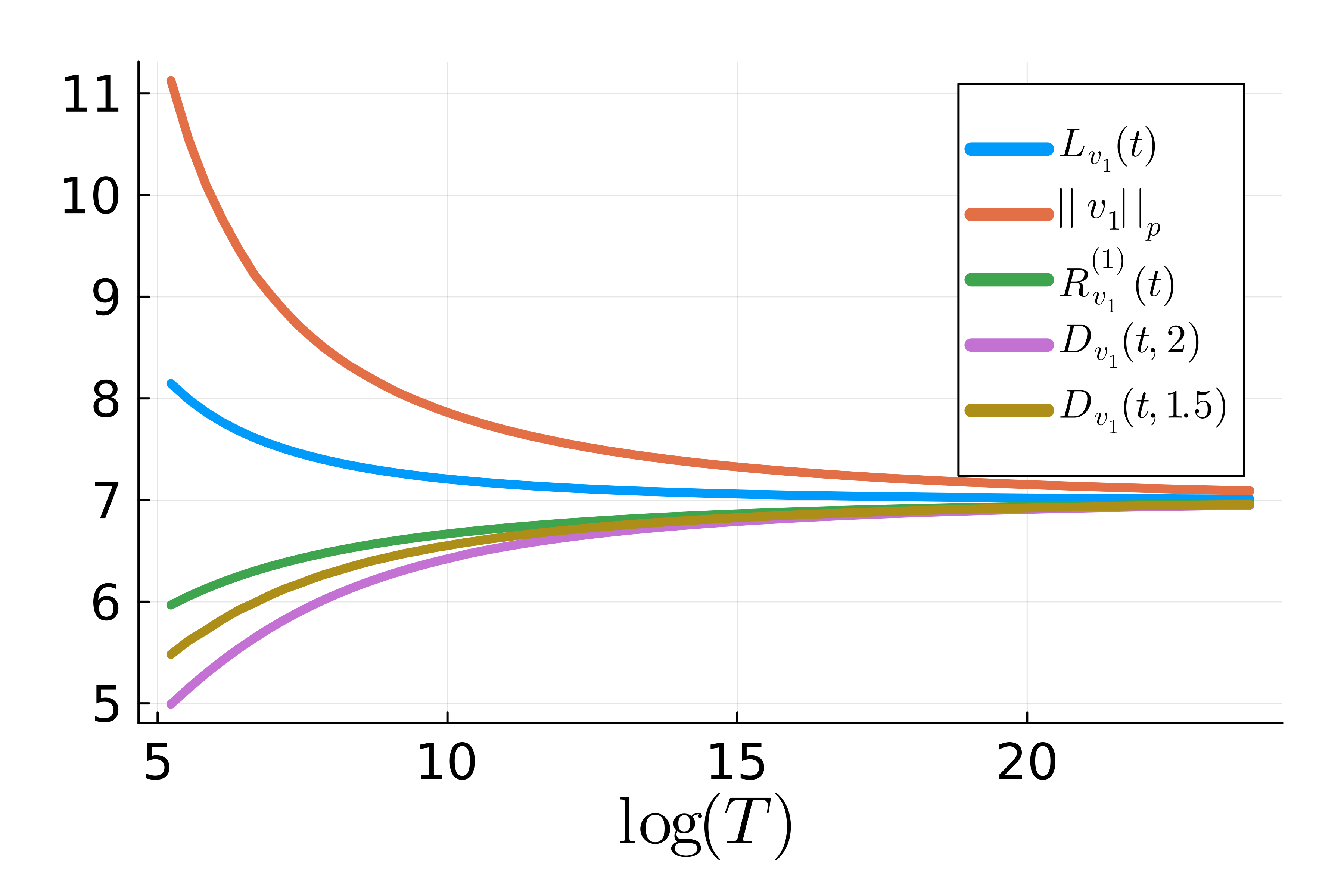}
\caption{A comparison of the smooth approximations $L_{v_1}(t),R_{v_1}(t),||v_1||_p, D_{v_1}(t,2)$, and $D_{v_1}(t,1.5)$ of the maximum of $v_1$, plotted against the natural logarithm of $T$, the required absolute value of floating point numbers for evaluation.}
\label{fig:Full_M_Comparison_v1}
\end{figure}
\begin{figure}
\includegraphics[scale=0.065]{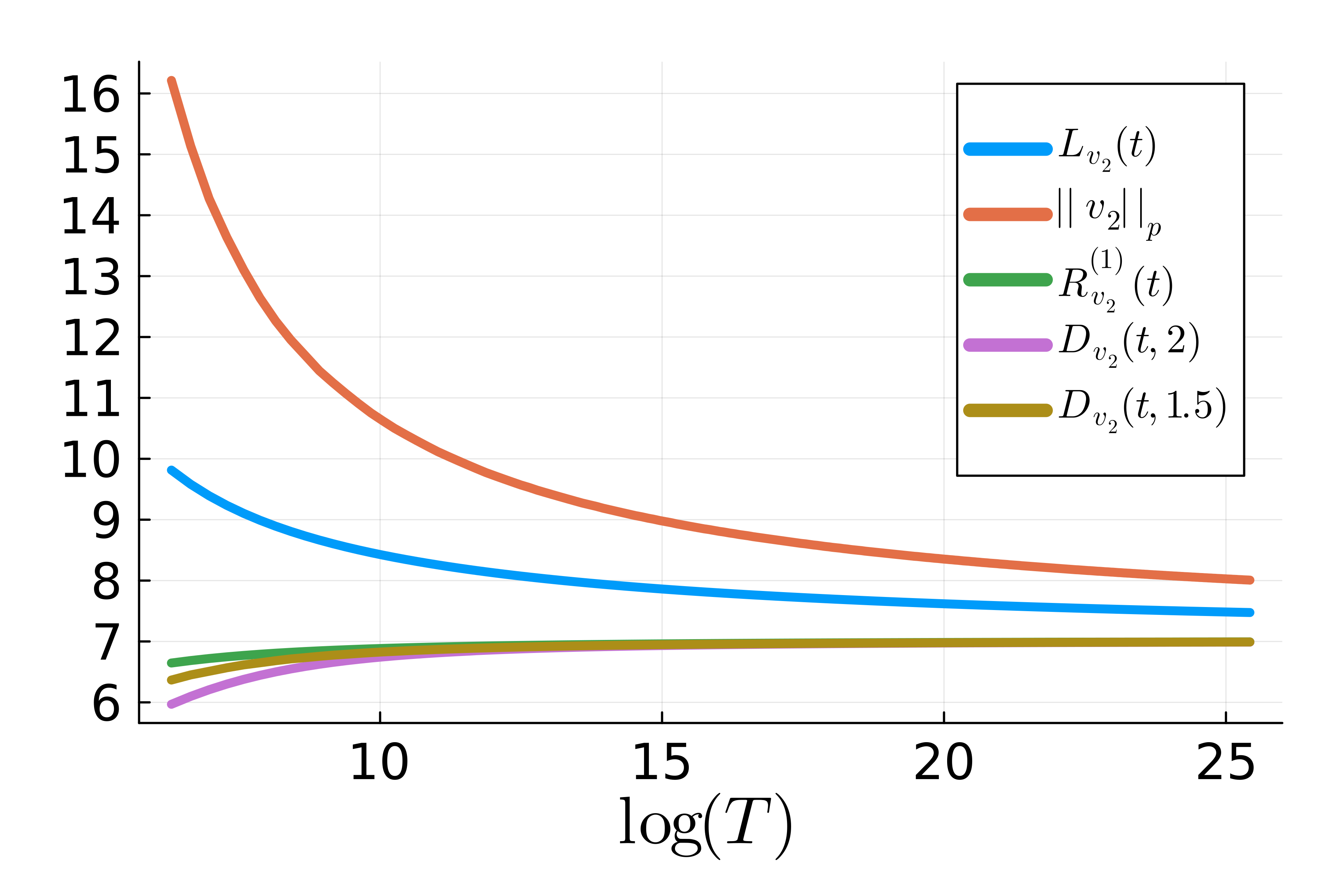}
\caption{A comparison of the smooth approximations $L_{v_1}(t),R_{v_1}(t),||v_1||_p, D_{v_1}(t,2)$, and $D_{v_1}(t,1.5)$ of the maximum of $v_2$, plotted against the natural logarithm of $T$, the required absolute value of floating point numbers for evaluation.}
\label{fig:Full_M_Comparison_v2}
\end{figure}

\end{example}

\subsection{Computing the multiplicity}
\label{secsec:ComputingMultiplicity}

Due to the simplistic nature of the expansion
in \eqref{eq:Lexpansions}
for $L_v(t)$, we consider
computing the multiplicity
$\mu_M$ for the maximum $M$.
In particular, it is easy to see from~\eqref{eq:Lexpansions}
that
\begin{equation}
\label{eq:MultiplicityAsLimit}
\mu_M = \lim_{t \to \infty} t^{L_v(t)-M} \end{equation}
Of course, using this expression requires \emph{a priori} knowledge of $M$ which can be
attained, for example, in the integer case by applying Theorem~\ref{thm:Ltbounds}.

\begin{theorem}
\label{thm:LtMultBounds}
Given $v \in \Z^n$, $\lfloor t^{L_v(t)-\lfloor L_v(t)\rfloor}\rfloor = \mu_M$
whenever
\[
t > \max\left\{n-\mu_M, \frac{\mu_M+\sqrt{\mu_M^2+4(n-\mu_M)}}{2} \right\}
\]
\end{theorem}
\begin{proof}
When $t>\frac{\mu_M+\sqrt{\mu_M^2+4(n-\mu_M)}}{2}$, {Theorem~\ref{thm:Ltbounds}} provides that $\lfloor L(t) \rfloor = M$.
Using a worst-case analysis, one has
$$0\leq t^{L_v(t)-M} - \mu_M \leq (n-\mu_M) t^{-1}$$
with the worst-case upper bound below $1$
when $t > n - \mu_M$.
\end{proof}

\begin{example}
Continuing with $v_1$ and $v_2$ from
Example~\ref{ex:simple_integer_multiplicity_difference},
Theorem~\ref{thm:LtMultBounds} provides
$t^{L_{v_1}(t)-\lfloor L_{v_1}(t)\rfloor} \in [1,2)$ for $t> 6$
and
$t^{L_{v_2}(t)-\lfloor L_{v_2}(t)\rfloor} \in [5,6)$ for $t> 6$
with
Figure~\ref{fig:multiplicityComputationCompare}
showing convergence in advance of such
worst-case~bounds.
\end{example}

\begin{figure}[!hb]
\includegraphics[scale=0.065]{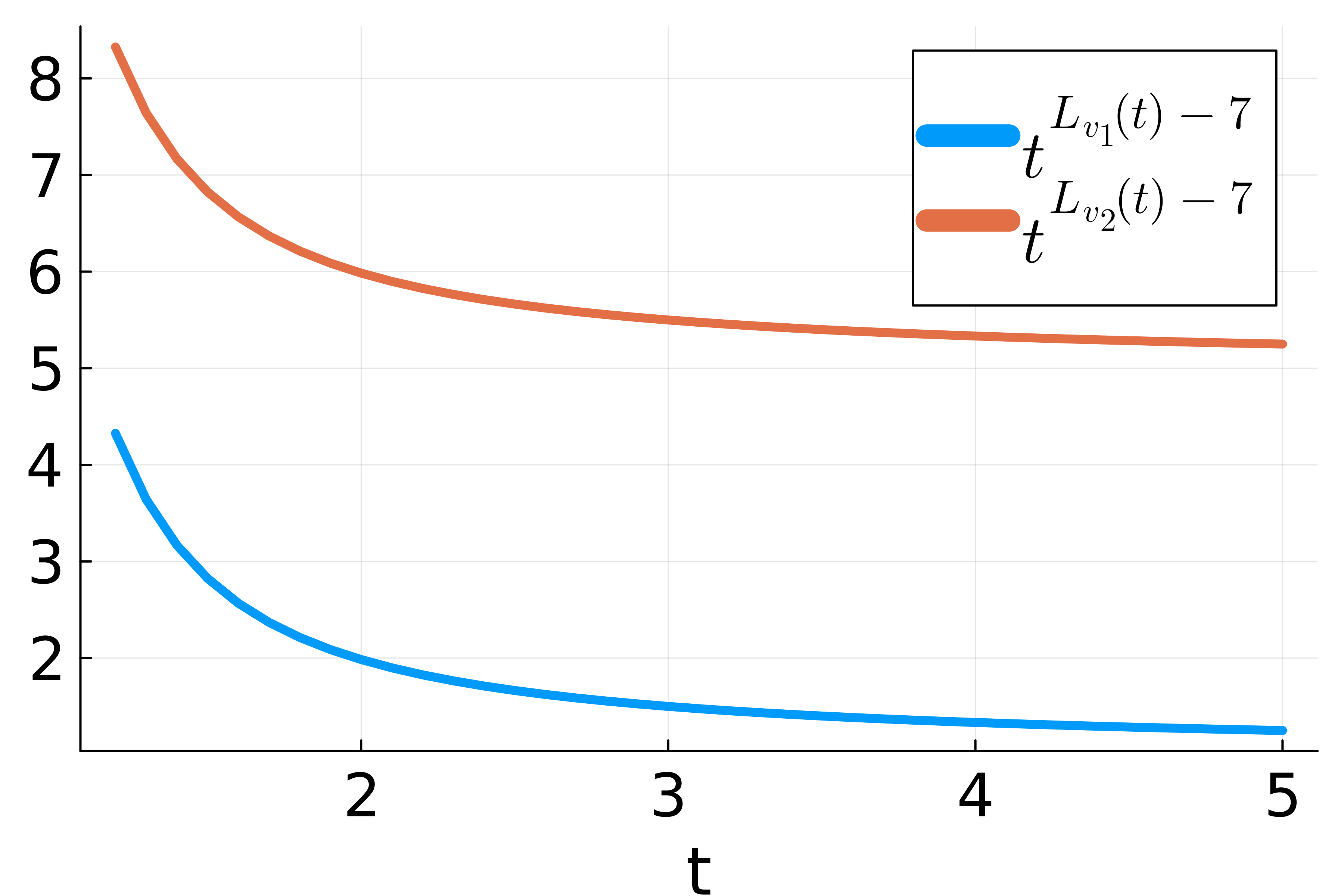}
\caption{The graphs of $t^{L(t)-M}$ applied to $v_1$ and $v_2$ from  Example~\ref{ex:simple_integer_multiplicity_difference}.
}
\label{fig:multiplicityComputationCompare}
\end{figure}

\subsection{Combining \texorpdfstring{$R^{(k)}_v(t)$}{Rkv(t)} to improve convergence and compute \texorpdfstring{$g_2$}{g2}}
\label{secsec:ComputingOther}
Since all of the higher-order derivatives $R_v^{(k)}(t)$ have the same convergence
rate, one can combine them in various
ways to increase the convergence rate
as well as extract other information
about $v$.  The following demonstrates
a higher-order approximation of~$M$
along with approximating $g_2$. The computation of $g_2$ and $M$ produces, as a byproduct, the second largest element of $v$, namely~\mbox{$w_2 = M-g_2$}.

\begin{theorem}
\label{thm:HigherRatioOtherInformation}
For $v \in \Z^n$, we have
\begin{equation}\label{eq:CombinedRM}
\frac{2R_v^{(1)}(t)R_v^{(3)}(t) -
R_v^{(2)}(t)\left(R_v^{(1)}(t) + R_v^{(2)}(t)\right)
}{
R_v^{(1)}(t) -
3R_v^{(2)}(t) +
2R_v^{(3)}(t)} = M + O(t^{-g_2-1})
\end{equation}
and
\begin{equation}\label{eq:CombinedRG2}
\frac{R_v^{(1)}(t) - 3R_v^{(2)}(t) +
2R_v^{(3)}(t)}{R_v^{(2)}(t) - R_v^{(1)}(t)}
= g_2 + O(t^{-1}).
\end{equation}
\end{theorem}
\begin{proof}
From Proposition~\ref{prop:LLexpansion} and \eqref{eq:RkExpansions},
$$\begin{array}{rcl}
R_v^{(1)}(t) &=& M - g_2\dfrac{\mu_{w_2}}{\mu_M}t^{-g_2} + O(t^{-g_2-1}), \\[0.15in]
R_v^{(2)}(t) &=& M - g_2(g_2+1)\dfrac{\mu_{w_2}}{\mu_M}t^{-g_2}+O(t^{-g_2-1}), \\[0.15in]
R_v^{(3)}(t) &=& M-\dfrac{g_2(g_2+1)(g_2+2)}{2} \dfrac{\mu_{w_2}}{\mu_M}t^{-g_2}+O(t^{-g_2-1})
\end{array}$$
and so the result follows by direct symbolic elimination.
\end{proof}

\begin{example}
We illustrate Theorem~\ref{thm:HigherRatioOtherInformation}
using $v_1$ and $v_2$ from Example~\ref{ex:simple_integer_multiplicity_difference}.
Figure~\ref{fig:RCombinedV1V2}
compares the convergence of
$R_v^{(1)}(t)$, $R_v^{(2)}(t)$, $R_v^{(3)}(t)$, and the combined formula in \eqref{eq:CombinedRM}
for $v=v_1$ and $v=v_2$ to $M=7$ for both.
For both cases, one sees faster convergence
as expected from \eqref{eq:CombinedRM}.
Additionally, Figure~\ref{fig:G2Approximations}
shows the convergence of the combined formula
in \eqref{eq:CombinedRG2}
for $v_1$ and $v_2$ to $g_2=1$ for both.

\begin{figure}[!htpb]
\includegraphics[scale=0.055]{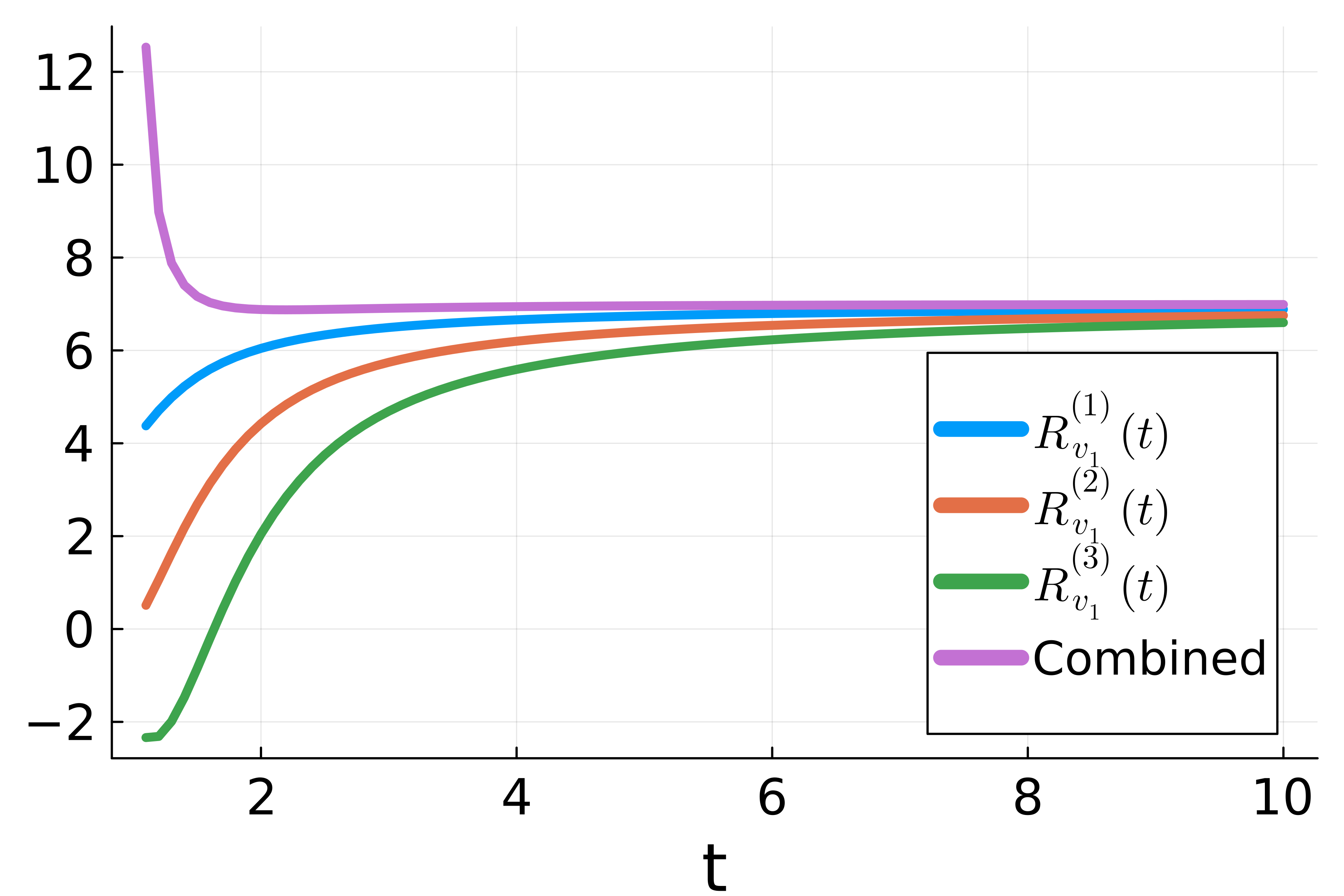} \quad \quad 
\includegraphics[scale=0.055]{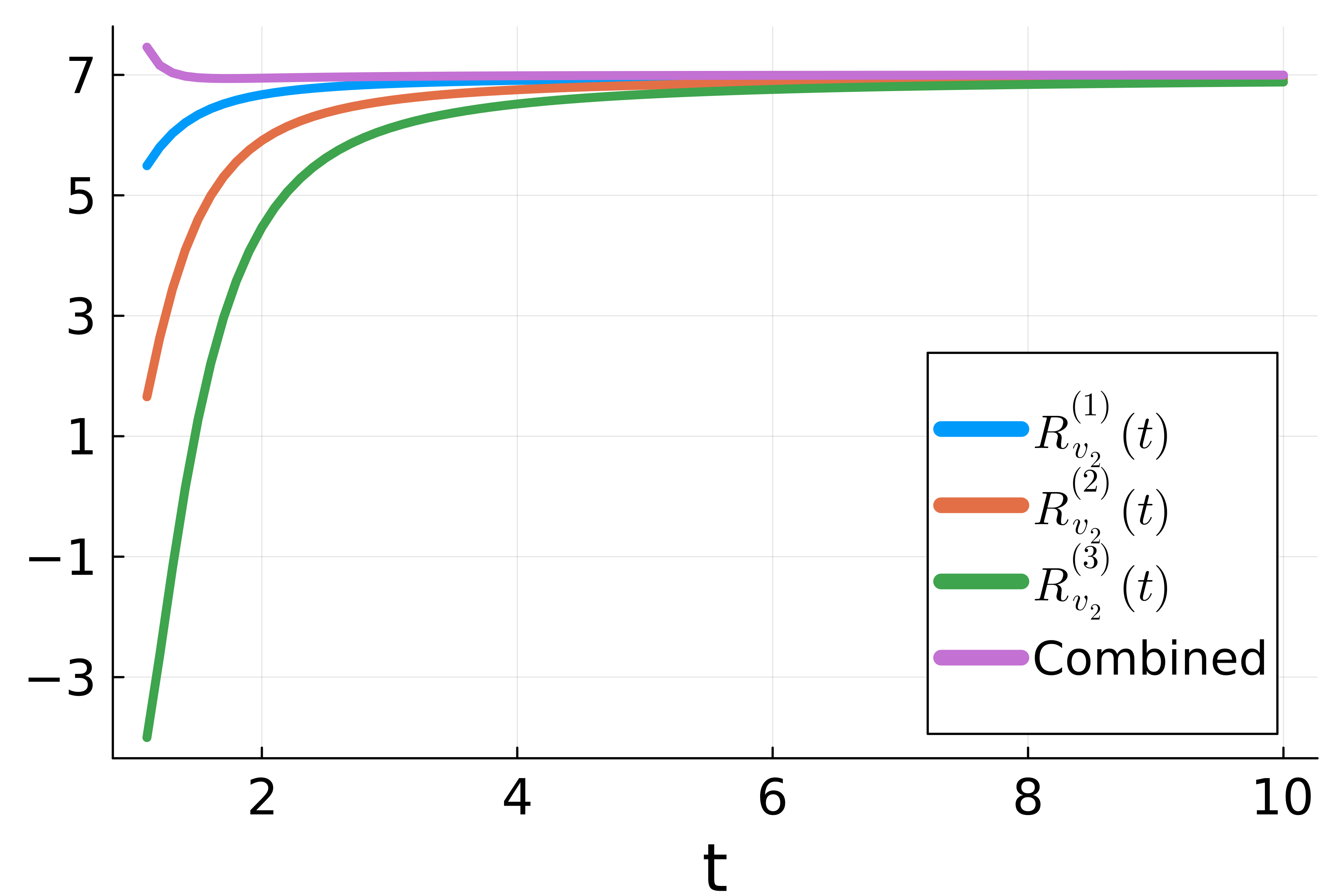}
\caption{Comparison of various methods
to approximate $M$ for $v_1$ and $v_2$.}
\label{fig:RCombinedV1V2}
\end{figure}

\begin{figure}[!htpb]
\includegraphics[scale=0.07]{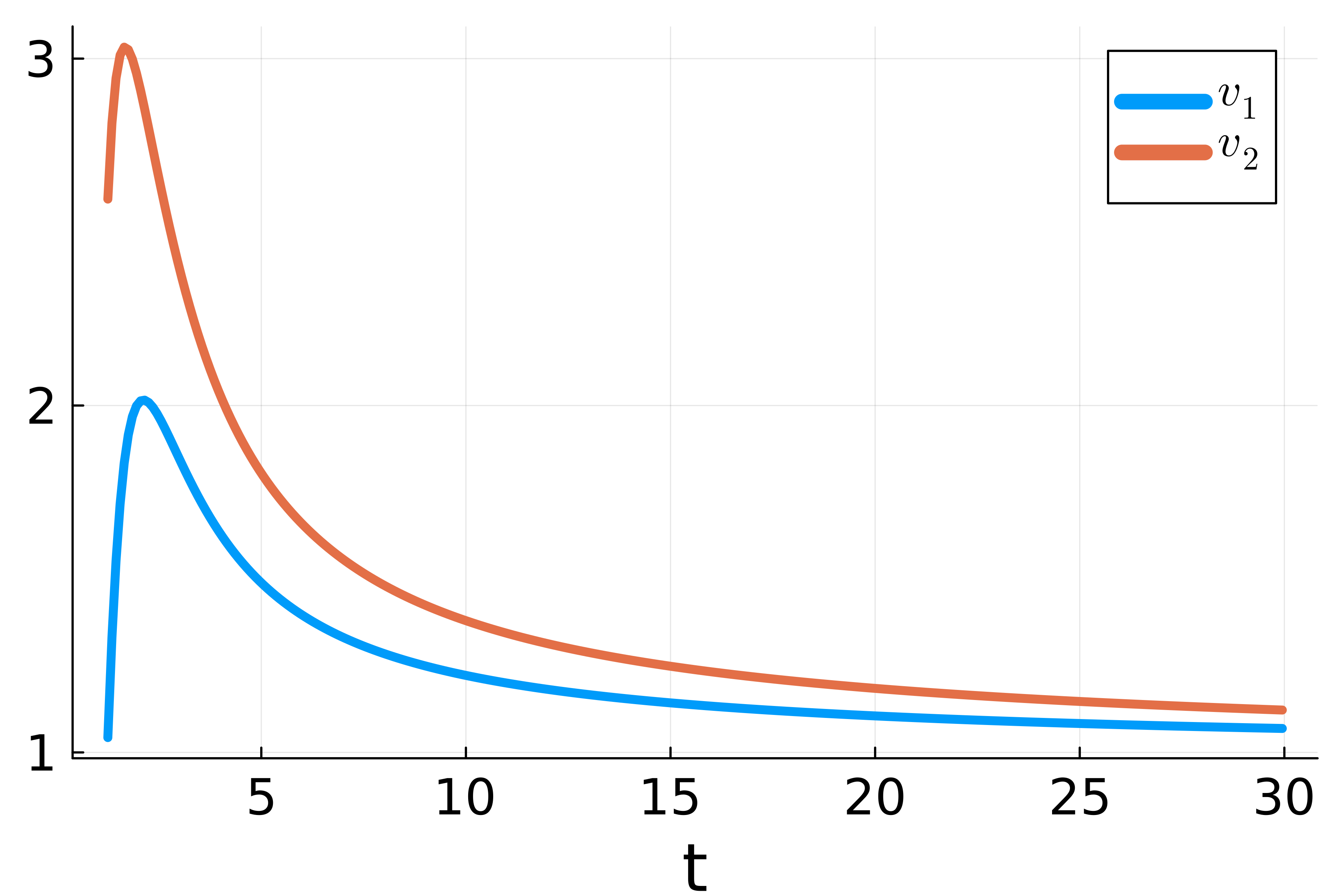}
\caption{The graphs of \eqref{eq:CombinedRG2}
for $v_1$ and $v_2$ which converge
to $g_2=1$ for both cases.}
\label{fig:G2Approximations}
\end{figure}

\end{example}

\begin{remark}
\label{rem:RintermsofDD}
The functions $R_v^{(k)}(t)$ are linear combinations of the derivatives $\mathcal D^{(k)}(t) = \frac{d^k}{du^k} \mathcal L_v(e^u)$, e.g.,
\begin{align*}
R^{(1)}(t) &= \mathcal D^{(1)}(t), \\
R^{(2)}(t) &= -\mathcal D^{(1)}(t) + \mathcal D^{(2)}(t), \\
R^{(3)}(t) &= \frac{2 \mathcal D^{(1)}(t) - 3 \mathcal D^{(2)}(t) + \mathcal D^{(3)}(t)}{2}, \\
R^{(4)}(t) &= \frac{-6\mathcal D^{(1)}(t) + 11 \mathcal D^{(2)}(t) -6 \mathcal D^{(3)}(t) + \mathcal D^{(4)}(t)}{6}.
\end{align*}
In particular, the linear transformation
that maps the first $r$ values of $\mathcal D^{(k)}(t)$ to the first $r$ values of $R^{(k)}(t)$ is represented by an $r \times r$ lower triangular matrix $A^{(r)}$.
For $j\leq i$, the $(i,j)$-entry of
$A^{(r)}$ is
$$\frac{(-1)^{i+1}}{(i-1)!}\mathcal S_{ij}$$ where $\mydef{\mathcal S_{ij}}$ is a Stirling number of the first kind. For example,
\[
A^{(4)} = \begin{bmatrix} (-1)^2  1 & 0 & 0 & 0 \\ (-1)^3 1 & (-1)^3 (-1) & 0 & 0 \\ \left(\frac{(-1)^4}{2}\right)  2 & \left(\frac{(-1)^4}{2}\right)  (-3) & \left(\frac{(-1)^4}{2}\right) 1 & 0  \\
\left(\frac{(-1)^5}{6} \right) 6 &\left(\frac{(-1)^5}{6} \right) (-11)&\left(\frac{(-1)^5}{6}  \right)6&\left(\frac{(-1)^5}{6} \right) (-1)
\end{bmatrix}.
\]
\end{remark}

\section{The tropical viewpoint}
\label{sec:Tropical}

We interpret
our previous results geometrically using tools from \emph{tropical geometry}.
We stress that throughout this section, we consider specific families of tropical constructions which exist more generally. Namely, the varieties we consider are graphs of univariate Laurent polynomials with positive coefficients.  Such varieties are quite special and so the results of this section may not hold in the more general setting. For an introduction to tropical geometry, we invite the interested reader to consult the standard reference \cite{MacSturm}.

To utilize the tropical geometry framework,
we assume throughout this section
that $v \in \mathbb{Z}^n$,
define $\mydef{\C^\times}=\C\backslash\{0\}$, and consider the function
\begin{align*}
\varphi: \C^\times &\to \mathbb{C}^2\\
t &\mapsto F_v(t).
\end{align*}
Let $\mydef{X_v}$ be the graph of $\varphi$ intersected with $(\C^\times)^2$. We note that $X_v$ is the set of zeros of the polynomial $\mydef{\mathcal F_v(t,y)} = y-F_v(t)$: $$X_v = \{(t,y) \in (\C^\times)^2 \mid \mathcal F_v(t,y) = 0\} \subset (\C^\times)^2.$$

The \emph{Newton polygon} of $\mathcal F_v(t,y)$ is the convex hull $\mydef{\mathcal N(\mathcal F_v(t,y))}$ of the exponent vectors of $\mathcal F_v(t,y)$. In this case, $\mathcal N(\mathcal F_v(t,y))$ is simply the triangle $\mydef{\Delta_v}$ with vertices $(M,0)$, $(\min(v),0)$, and $(0,1)$
as illustrated in Figure~\ref{fig:NewtonAmoebaTrop}(a).
The union of the outer normal rays of $\Delta_v$ along with the origin form a \emph{polyhedral fan} called the \mydef{tropicalization} of $X_v$,  denoted $\mydef{\textrm{trop}(X_v)}$
and illustrated in Figure~\ref{fig:NewtonAmoebaTrop}(c).
The fan $\textrm{trop}(X_v)$ is a \emph{tropical curve} which encodes the asymptotic behavior of $X_v$ near the coordinate axes $\mathbb{C}^2 \backslash (\mathbb{C}^\times)^2$. Note that in our specific situation, the Newton polytope $\mathcal N(\mathcal F_v(t,y))$, and hence the tropical curve $\textrm{trop}(X_v)$, depends only on $\min(v)$ and $\max(v)$.

\begin{figure}[!htpb]
\includegraphics[scale=0.08]{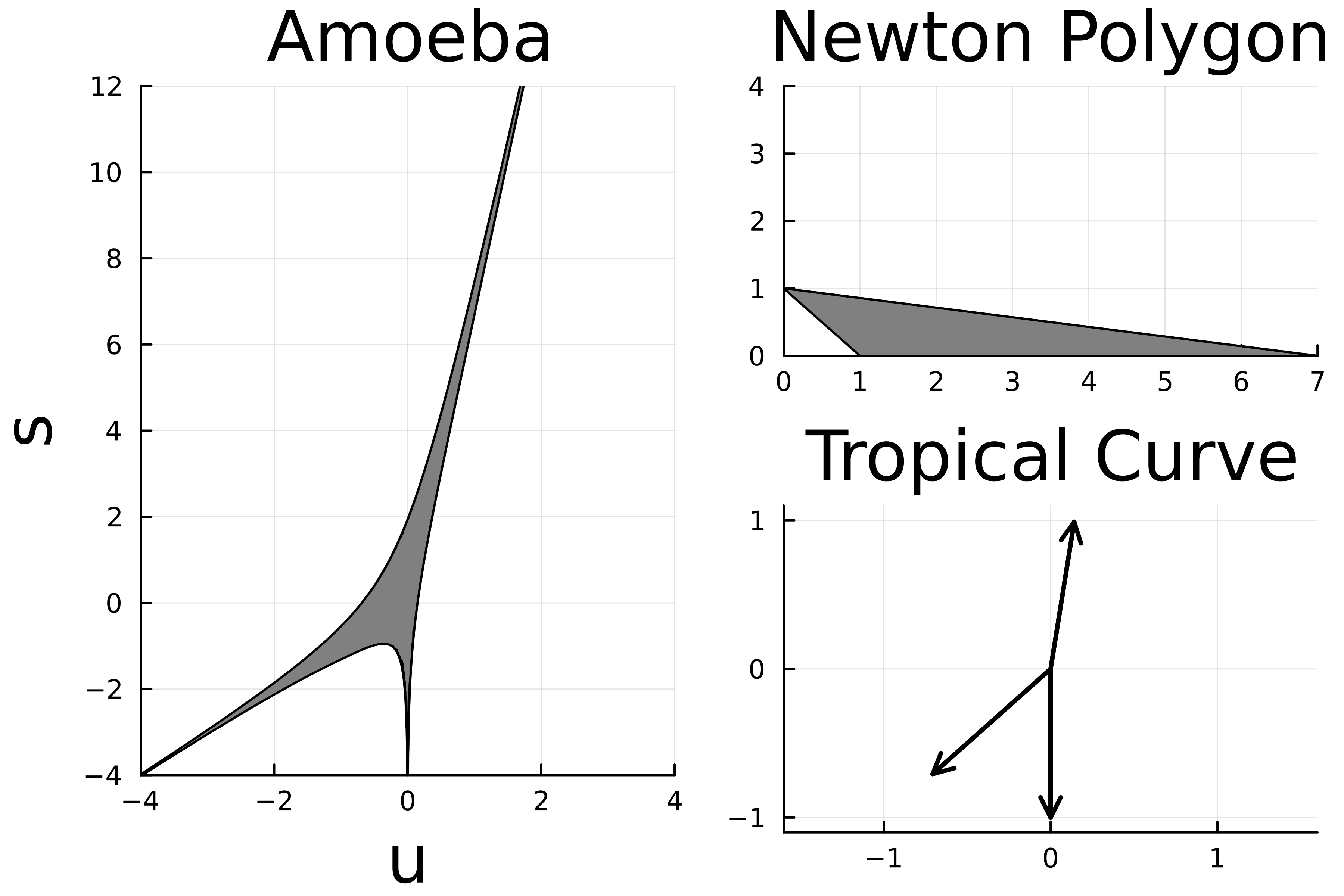}
\caption{(a) The Newton polygon of $\mathcal F_{v_1}(t,y)$ where $v_1$ is as in Example \ref{ex:simple_integer_multiplicity_difference} (b) The amoeba $\mathcal A(X_{v_1})$ (c) The tropical variety $\trop(X_{v_1})$.}
\label{fig:NewtonAmoebaTrop}
\end{figure}

An alternative construction of $\textrm{trop}(X_v)$, due to Bergman \cite{Bergman}, involves the image $\mydef{\mathcal A_\tau(X_v)}$ of $X_v$ under the log-absolute value map:
\begin{align*}
\mydef{\textrm{Log}_\tau|\cdot |}: (\C^\times)^2 &\to \mathbb{R}_{u,s}^2 \\
(t,y) &\mapsto (\log_\tau(|t|),\log_\tau(|y|)).
\end{align*}
The set $\mathcal A_\tau(X_v)$ is called the \mydef{$\tau$-amoeba} of $X_v$. We remark that we use $\mydef{u}$ and $\mydef{s}$ for coordinates of the codomain and that the overlap of the symbol $u$ with previous sections is intentional.  Undecorated, the notation $\mathcal A(X_v) \subseteq \mathbb{R}^2_{u,s}$ refers to the $e$-amoeba of $X_v$
as illustrated in Figure~\ref{fig:NewtonAmoebaTrop}(b).  Since $\mathcal A_\tau(X_v) = \frac{1}{\log(\tau)}\mathcal A(X_v)$, the set $\mathcal A(X_v)$ contains all of the information about all of the amoebas of $X_v$. Since the absolute values of coordinates of points in $X_v$ may be arbitrarily large or small, the set $\mathcal A(X_v)$ is unbounded. The portions which approach infinity are loosely referred to as the \emph{tentacles} of the amoeba.
As $\tau \to \infty$, these tentacles limit to the rays of $\textrm{trop}(X_v)$. In this sense, $\textrm{trop}(X_v)$ contains asymptotic information about $X_v$ and is sometimes referred to as the \emph{logarithmic limit set} of the
amoeba of~$X_v$.

We call lines which intersect an amoeba $\mathcal A(X_v)$ in a ray \mydef{tentacle lines}. Up to translation, these rays are exactly those in $\textrm{trop}(X_v)$. Note that many tentacle lines may be associated to the same tropical ray (see the vertical rays of Figure \ref{fig:tentacleLines}).
The following elementary facts relate the Newton polygon $\mathcal N(\mathcal F_v(t,y))$, amoeba $\mathcal A(X_v)$, and tropical curve $\textrm{trop}(X_v)$. We encourage the reader to refer to Figure \ref{fig:tentacleLines}. These facts are specializations of a more general relationship between these three objects (see \cite{MacSturm} for more details).
\begin{enumerate}
\item[(1)] The tentacle lines of $\mathcal A(X_v)$ corresponding to the ray in $\textrm{trop}(X_v)$ spanned by $(0,-1)$ correspond to distinct moduli of complex roots of $F_v(t)$.
\item[(2)] When the lowest order term of  the Laurent polynomial $F_v(t)$ is a constant $c$, $\trop(X_v)$ contains the ray spanned by $(-1,0)$. There is one tentacle line of $\mathcal A(X_v)$ associated to that ray, which occurs at height $\log(c)$.
\end{enumerate}

Additionally, the following facts relate $\mathcal A(X_v)$ and the function $u \mapsto \mathcal L_{v}(e^u)$.
\begin{enumerate}
\item[(3)] The upper boundary $\mydef{\mathcal U}$ of $\mathcal A(X_v)$ is the graph of $\mathcal L_v(e^u)$.
\item[(4)] $L_v(t)=L_v(e^u) = \frac{\mathcal L_v(e^u)}{u} $ is the slope of the ray from the origin to the point $(u,\mathcal L_v(e^u)) \in \mathcal A(X_v)$.
\item[(5)] $\mathcal D^{(k)}_v(e^u) = \frac{d^k}{du^k} \mathcal L_v(e^u)$ is the $k^{\rm th}$ derivative of the function $u \mapsto \mathcal L_v(e^u)$.
\item[(6)] $\mathcal D_v^{(1)}(e^u) = R^{(1)}_v(e^u)$ is the slope of the tangent line to the boundary of the amoeba at $(u,\mathcal L_v(e^u))$.
\end{enumerate}
We point out that $(3)$
follows from the fact that all of the coefficients of $F_v(t)$ are non-negative. In general, describing the boundaries of amoebas is challenging \cite{DeWolff13}.

\begin{proposition}
\label{prop:tentacleLines}
Let $v \in \mathbb{Z}^n$. The tentacle lines of $\mathcal A(X_v)$ are given by
\begin{enumerate}
\item[(a)] $s = \log(\mu_M) + M \cdot u$
\item[(b)] $s = \log(\mu_{m})+m \cdot u$ where $m=\min(v)$.
\item[(c)] $u = \log(|\xi_i|)$ where $\{\xi_i\}_{i=1}^d\subset \mathbb{C}$ are the roots of $F(t)$.
\end{enumerate}
\end{proposition}
\begin{proof}
As observed already, the tentacles in the $(0,-1)$ direction correspond to roots $\xi_i$ of $F_v(t)$ and they occur at $u=\log(|\xi_i|)$ which establishes (c). To see (a) and (b), suppose $m=\min(v)=0$. Then, $F(t)$ has a constant term of $\mu_m$ and hence
$\lim_{t\to 0} \log(F(t)) = \log(\mu_m)$. This limit indicates that $\mathcal A(X_v)$ has a horizontal tentacle occurring at $s = \log(\mu_m)$. However, translating $v$ by $a \in \Z$ amounts to sheering the Newton polygon space by $(\alpha,\beta) \mapsto (\alpha, a\alpha+\beta)$ and the amoeba space by $(u,s) \mapsto (u-as,s)$. In particular, this transformation does not change the $s$-intercepts of the tentacle lines. Hence, the tentacle line associated to the minimum of $v$ has the equation $s=\log(\mu_m)+m \cdot u$ whereas the line associated to the maximum $M$ has the equation $s=\log(\mu_M)+M \cdot u$.
\end{proof}

\begin{figure}[!b]
\includegraphics[width=12cm,height=8cm]{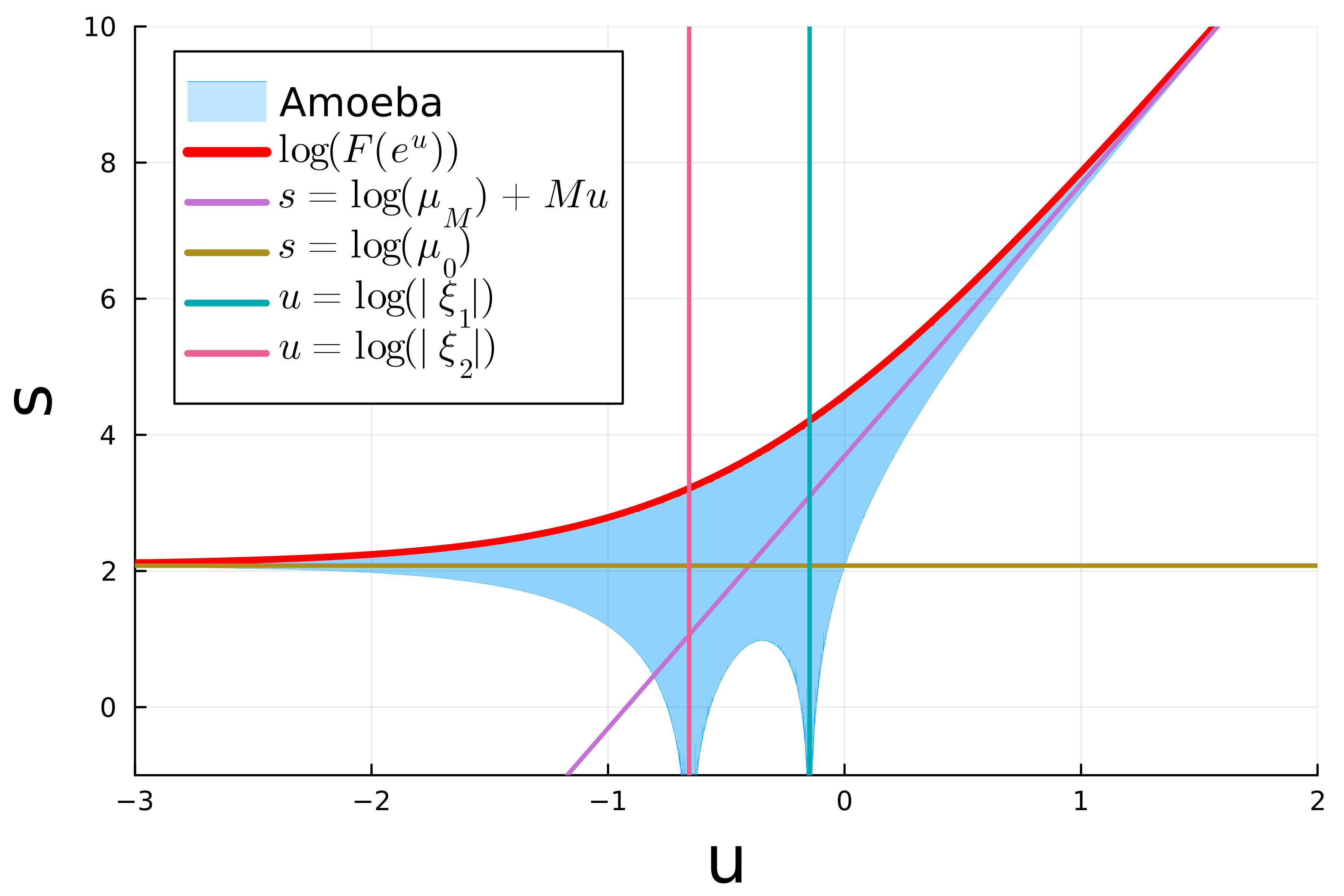}
\caption{The amoeba $\mathcal A(X_v)$ of the graph $X_v$ where $v$ is as in Example \ref{ex:TropicalExample} along with its tentacle lines.  }
\label{fig:tentacleLines}
\end{figure}
\begin{figure}[!htpb]
\includegraphics[width=12cm,height=8cm]{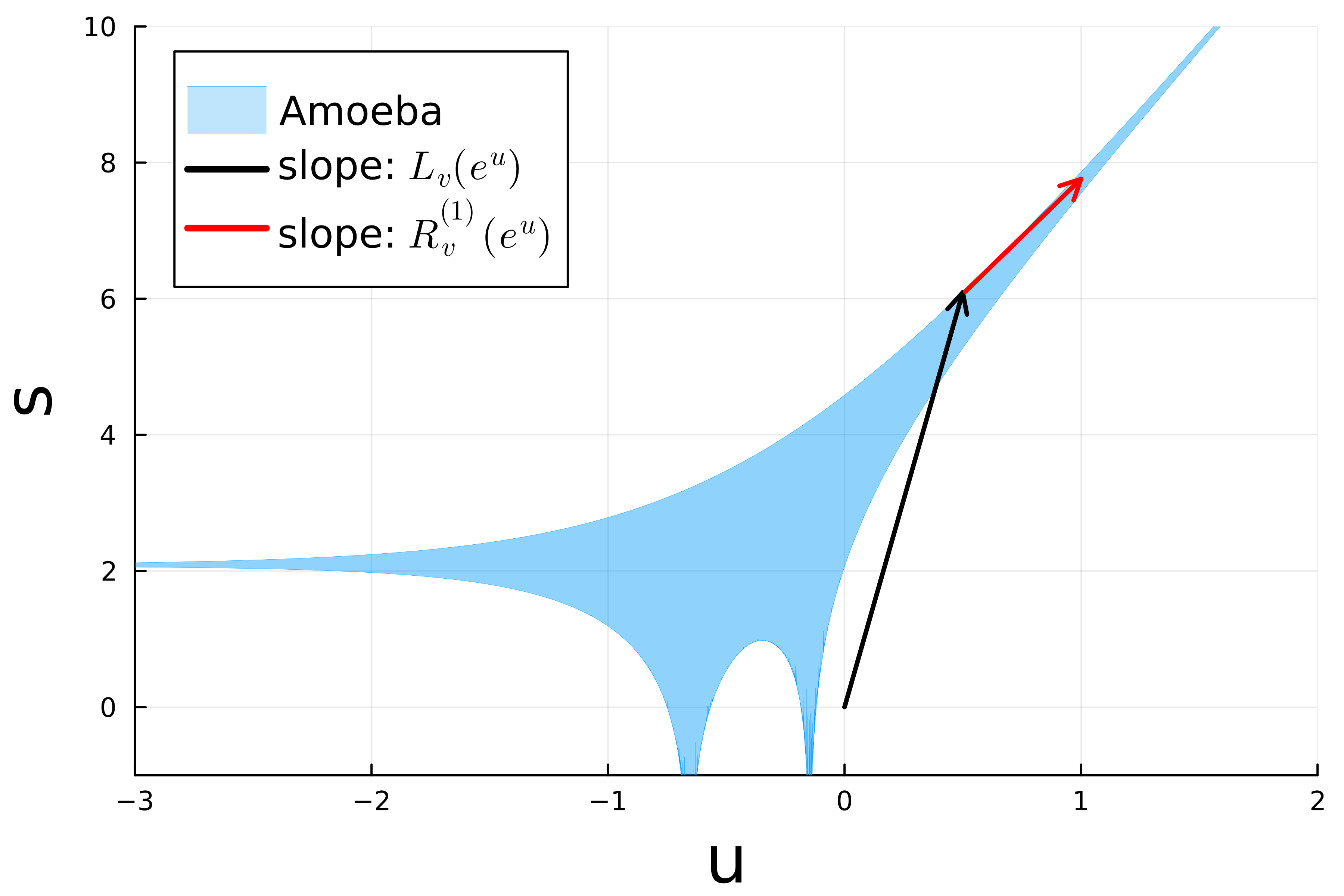}
\caption{The amoeba $\mathcal A(X_v)$ of the graph $X_v$ for $v$ as in Example \ref{ex:TropicalExample} along with the ray from the $(0,0)$ to $(u,L_v(e^u))$ and the ray from $(u,L_v(e^u))$ with slope $R^{(1)}_v(e^u) = \mathcal D^{(1)}_v(e^u)$.  }
\label{fig:geometricLvsR}
\end{figure}

Proposition \ref{prop:tentacleLines} gives a geometric interpretation of how the
multiplicity of $M$ in $v$ contributes to a slower convergence rate of $L_v(t)$ but not for $R_v(t)$: multiplicity corresponds to a translation of the tentacle line of~$\mathcal A(X_v)$ associated to the tropical ray spanned by $(1,M)$. This is geometrically displayed in Figure \ref{fig:geometricLvsR}.
\begin{example}
\label{ex:TropicalExample}
Let $v=[0_8,1_5,2_{40},3_5,4_{40}]$, where a subscript indicates multiplicity. The amoeba $\mathcal A(X_v)$ is shown in Figure \ref{fig:tentacleLines}. The polynomial $$F_v(t) = 8+5t+40t^2+5t^3+40t^4$$ has two pairs $\xi_1,\bar{\xi_1}$ and $\xi_2,\bar{\xi_2}$ of conjugate roots. Hence, there are two vertical tentacle lines of $\mathcal A(X_v)$. The tentacle line corresponding to the minimum multiplicity $\mu_m=\mu_0=8$ is horizontal at height $\log(8)$ and the equation of the remaining tentacle line is $s=\log(40)+4u$. The upper boundary of the amoeba is the image of the positive part of $X_v$ under the log-absolute value map.

Figure \ref{fig:tentacleLines} illustrates a geometric interpretation of the values of $L_v(e^u)$ and $R^{(1)}_v(e^u)$ as slopes of rays.

\end{example}

One may also interpret the Cauchy integral in Remark \ref{rem:Cauchy}:
$$
\frac{1}{2\pi \sqrt{-1}} \oint_{|t|=r} t^{-1}\cdot R_v^{(k)}(t^{-1})\cdot dt = M,$$
 in terms of tropical geometry when $k=1$. This is done via the \mydef{order map}:
\begin{align*}
\mydef{\textrm{ord}}:\mathbb{R}^2 &\to \mathbb{R}^2 \\
(u,s) &{\tiny{ \mapsto \left(\frac{1}{(2\pi\sqrt{-1})^2} \int_{\substack{\textrm{Log}|t|=u\\ \textrm{Log}|y|=s}} \frac{tF_v'(t)}{y-F_v(t)}\cdot \frac{dt dy}{t y}, \frac{1}{(2\pi\sqrt{-1})^2} \int_{\substack{\textrm{Log}|t|=u\\ \textrm{Log}|y|=s}} \frac{y}{y-F_v(t)}\cdot\frac{dt dy}{t y} \right).}}
\end{align*}
The function $\textrm{ord}$ is constant and $\mathbb{Z}^2$-valued on connected components of the complement $\mathbb{R}^2 \backslash \mathcal A(X_v)$ of the amoeba \cite{FPT2000}. In fact, $\textrm{ord}$ maps these components to distinct integer points in the Newton polygon $\Delta_v$. In particular, for any point $(u,s)$ in the bottom right complement component, the first integral
$$\frac{1}{(2\pi\sqrt{-1})^2} \int_{\substack{\textrm{Log}|t|=u\\ \textrm{Log}|y|=s}} \frac{tF_v'(t)}{y-F_v(t)}\cdot \frac{dt dy}{t y}$$
degenerates as $|y| \to 0$ (or $s \to -\infty$) to the integral
$$\frac{1}{2\pi\sqrt{-1}} \oint_{\substack{\textrm{Log}|t|=u}} \frac{tF_v'(t)}{F_v(t)} \cdot \frac{dt}{t} = \frac{1}{2\pi\sqrt{-1}} \oint_{\substack{\textrm{Log}|t|=u}} t^{-1} R_v^{(1)}(t)\cdot  dt = M$$
for sufficiently large $u$
as in \eqref{eq:CauchyIntegral}. The second integral, on the other hand, evaluates to zero. Since this value of $\textrm{ord}$ on $(u,s)$ is constant on connected components of the complement of the amoeba, this shows that $\textrm{ord}(u,s)$ evaluates to the vertex $(M,0)$ of $\Delta_v$ when $(u,s)$ is in the bottom-right component of $\mathbb{R}^2 \backslash \mathcal A(X_v)$,
in agreement with Remark \ref{rem:Cauchy}.

\section{Experiments}
\label{sec:Experiments}
\begin{figure}
\includegraphics[scale=0.421]{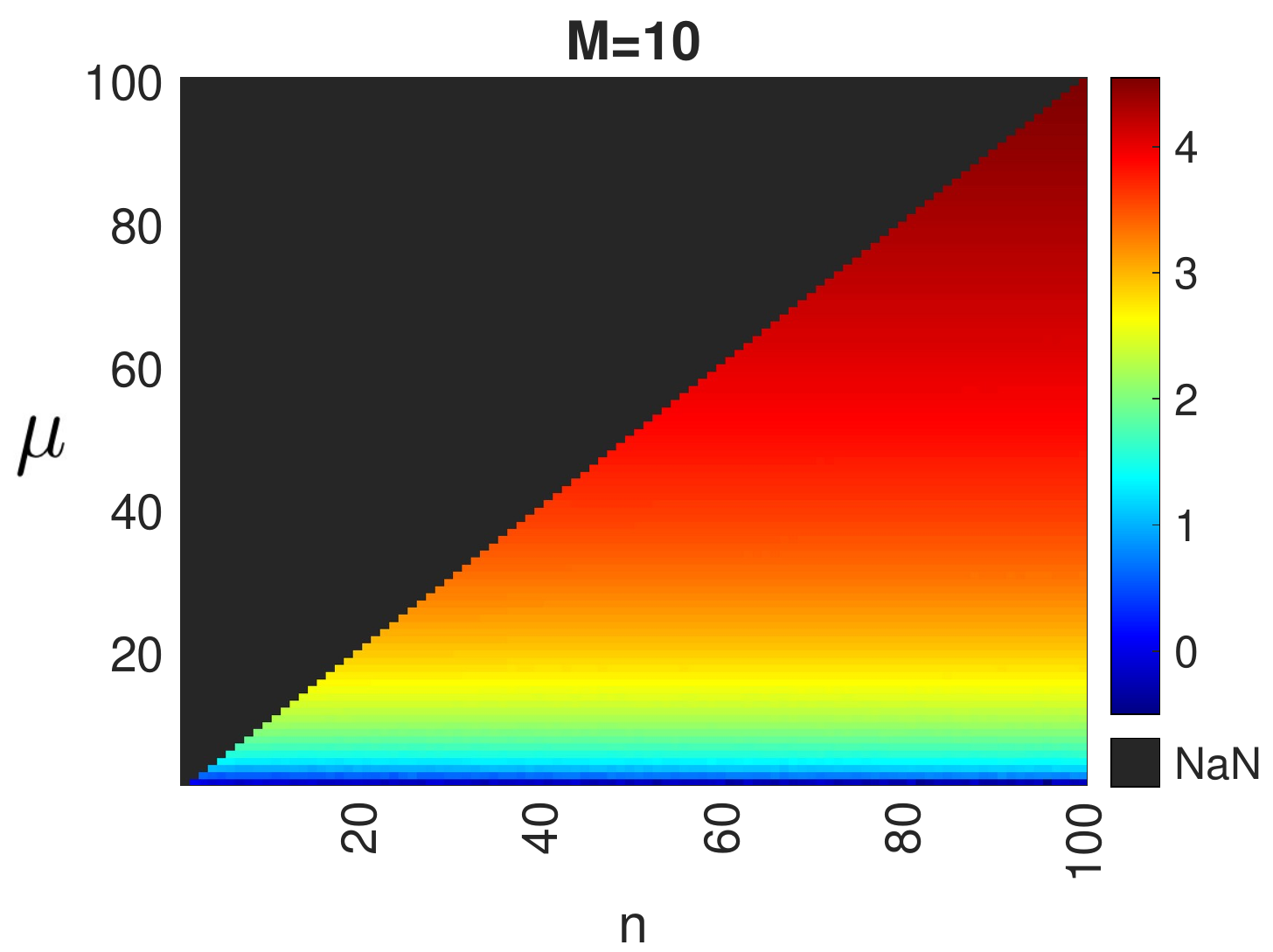}
\includegraphics[scale=0.421]{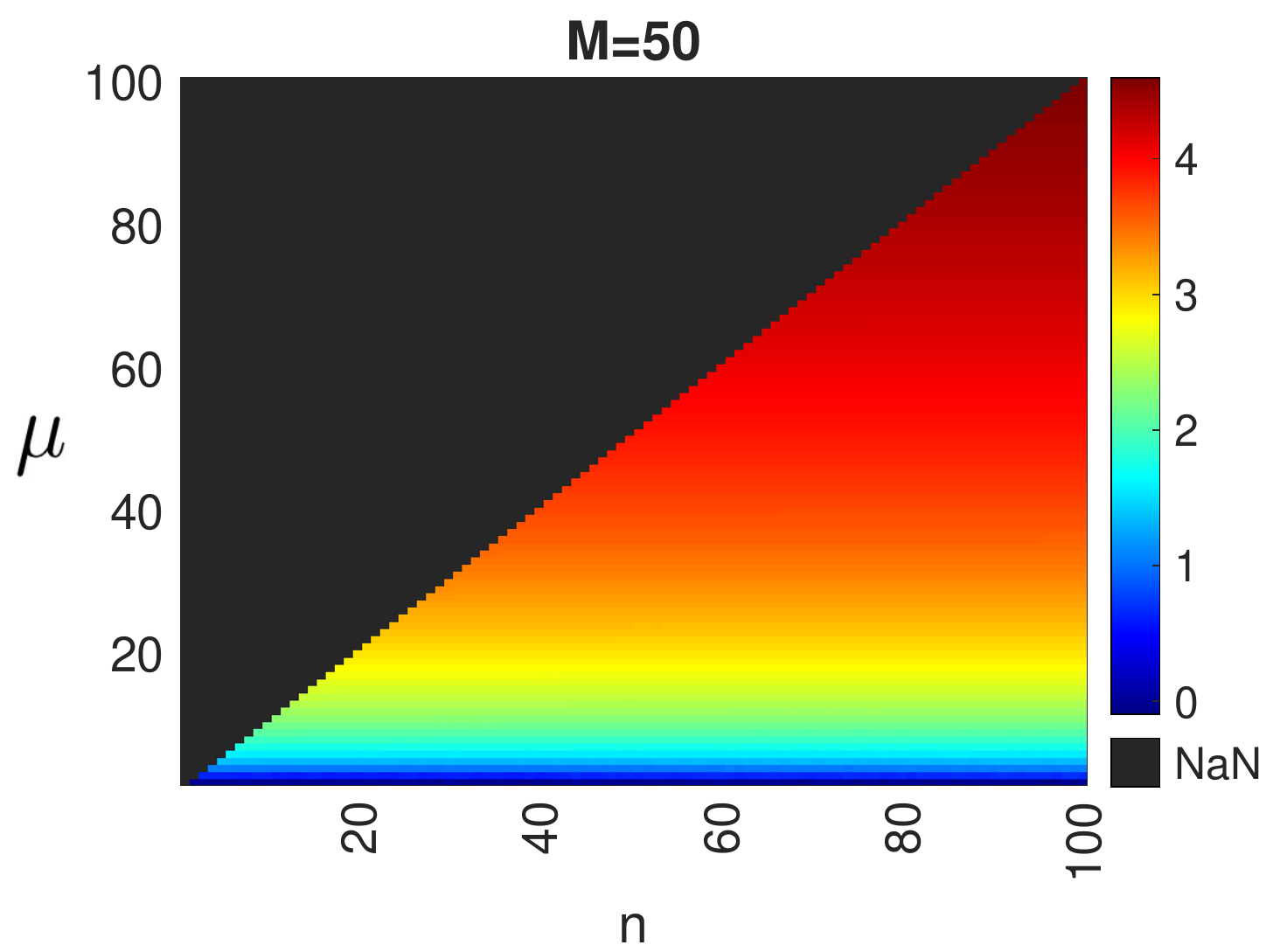}
\includegraphics[scale=0.421]{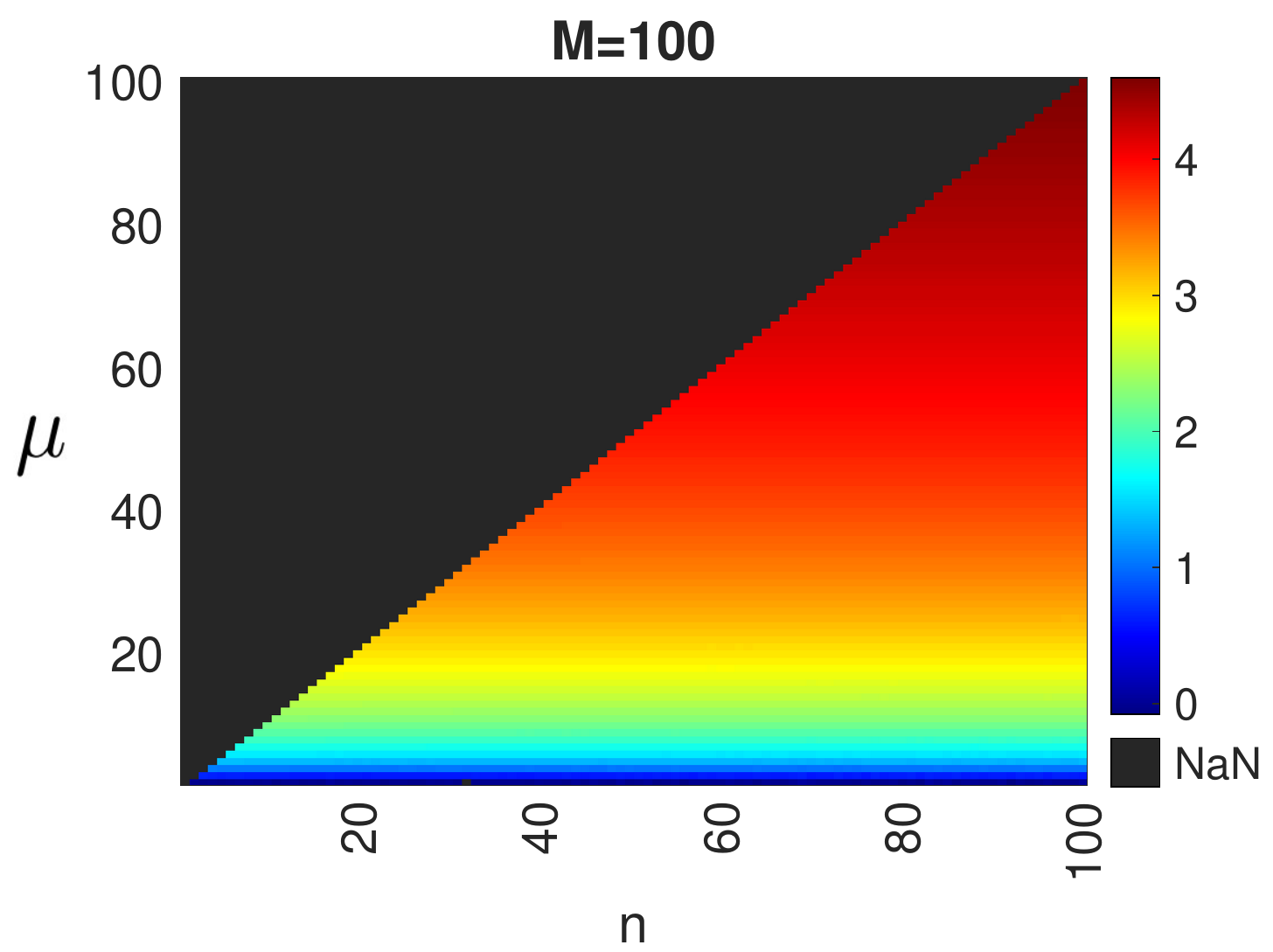}
\includegraphics[scale=0.421]{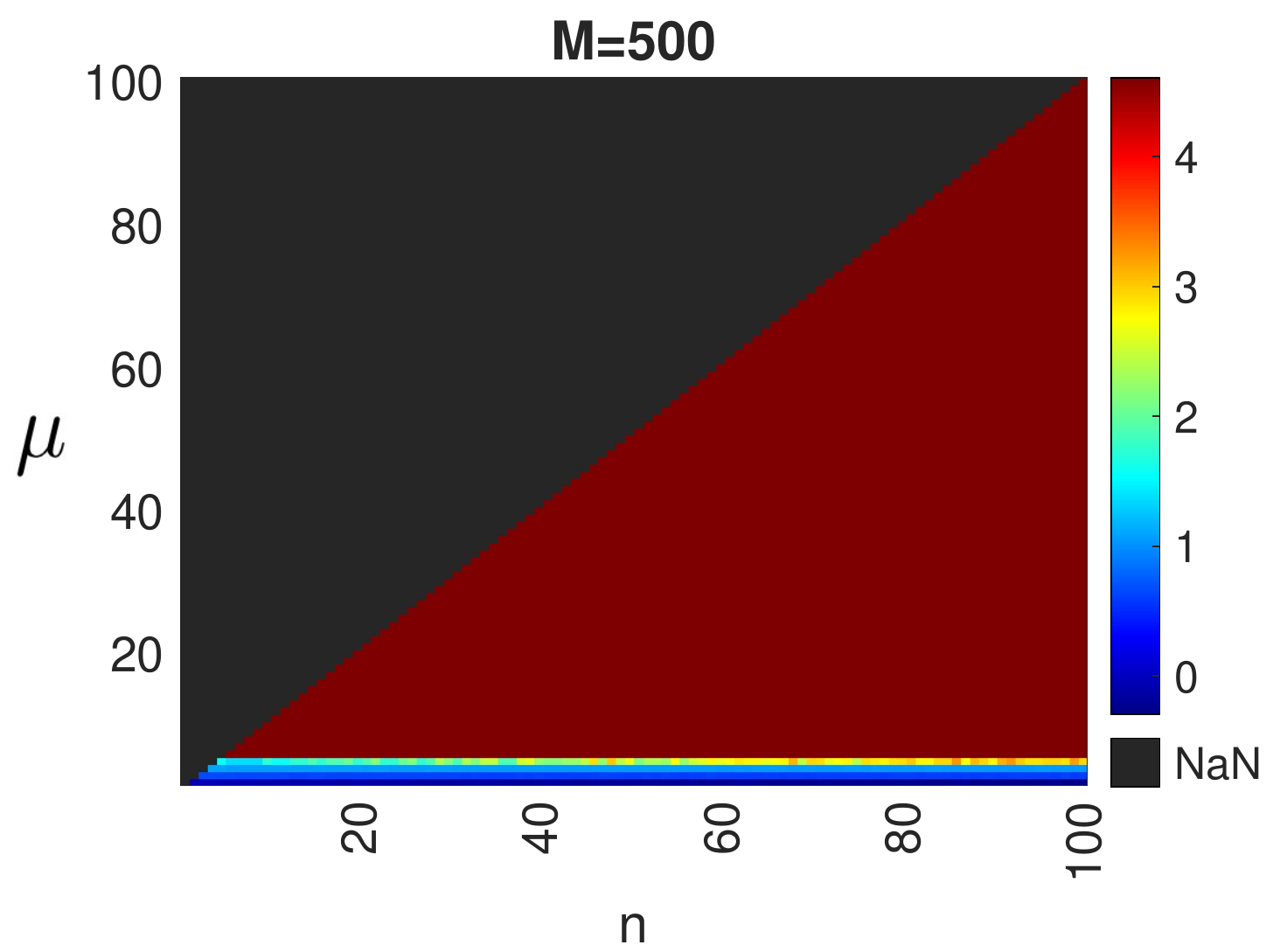}
\caption{Average results of the integer-valued vector $v$ experiments. The $x$-axis and $y$-axis correspond to $n=1,\ldots,100$ and $\mu=1,\ldots,n$ respectively. Each pixel is the average of $100$ subexperiments measuring $\log(1+t^*_{L_v}-t^*_{R_v})$ where $t^*_{L_v}$ and $t^*_{R_v}$ are the $t$-values such that the absolute error of the corresponding function is less than $1$. The maximum, $M$, are: (a) 10 (b) 50 (c) 100 (d) 500. Black pixels above the diagonal
are when $\mu > n$.}
\label{fig:experiment_integers}
\end{figure}
We compare each of the approximations discussed on a gallery of qualitatively different inputs $v$. In each of the following sections, we sample vectors $v$ from some prescribed distribution. We then compare the approximations of $M$ on these samples on average.

\subsection{Integer numbers}
We compared $L_v(t)$ and $R_v(t)$ in the integer case by defining a maximum $M$ with multiplicity $\mu$ for an integer-valued vector $v\in [1,M]^n$. That is, we defined $v$ of varying length $n=1,\dots,100$ with $M$ appearing $\mu$ times, where $\mu\leq n$ and the remaining $n-\mu$ values of $v$ were random integers sampled from $[1, M-1]$. Figure \ref{fig:experiment_integers} displays the four experiments comparing the performance of $L_v(t)$ and $R_v(t)$ for approximating a given maximum $M$, namely $M=10, 50, 100,$ and $500$, respectively. Performance is measured by $\mydef{t^*}$, the first value of $t$ to approximate $M$ up to an absolute error of $1$, so that the maximum is obtained from $L_v(t^*)$ or $R_v(t^*)$ through use of the floor or ceiling function, respectively.  Each experiment consists of $100$ subexperiments averaged over $\log(1+t^*_{L_v}-t^*_{R_v})$. The plotted values are logarithmic and offset by $1$ since, if $t^*_{L_v}=t^*_{R_v}$, then $\log(1+t^*_{L_v}-t^*_{R_v}) = \log(1) = 0$.

\begin{figure}[!htpb]
\includegraphics[scale=0.425]{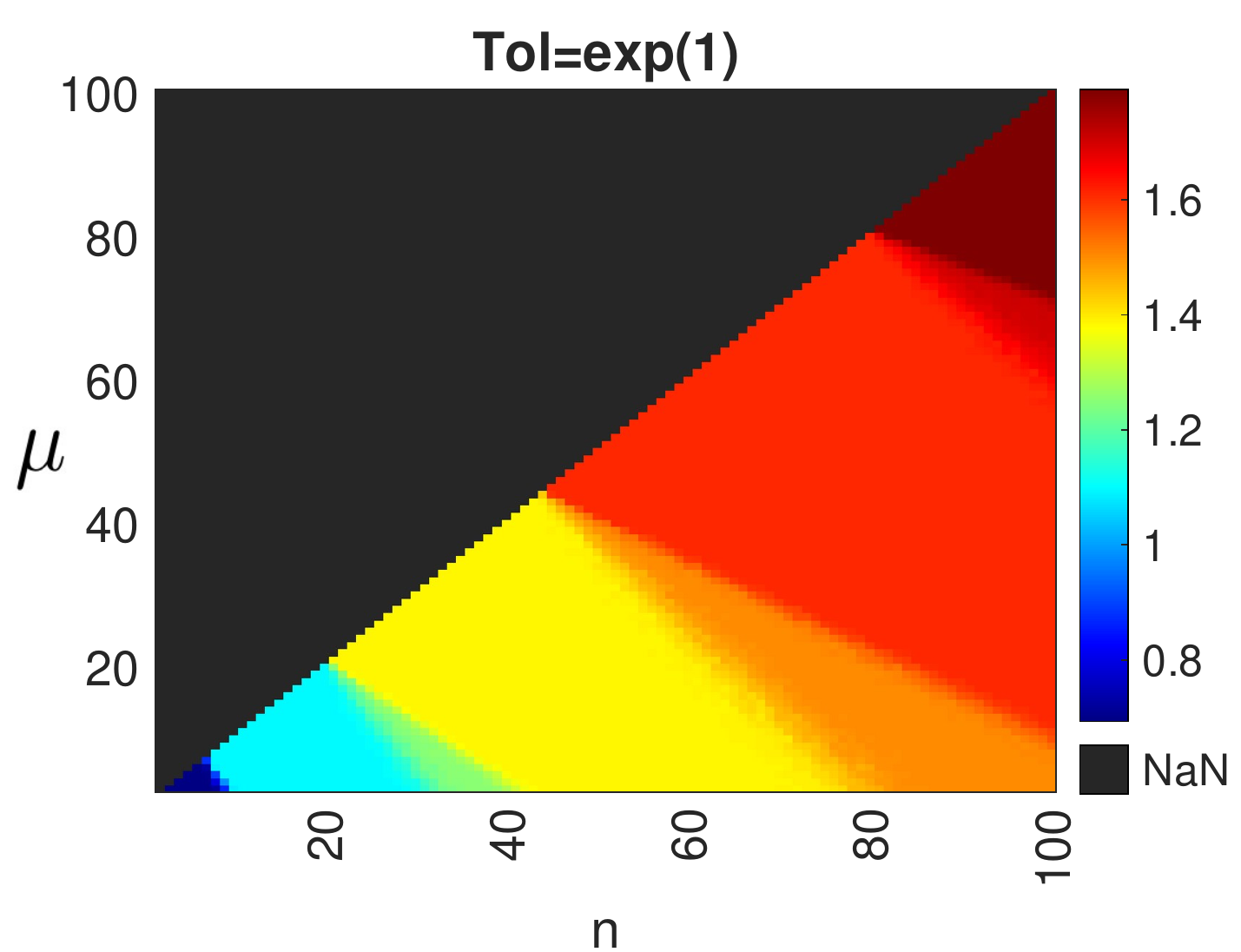}
\includegraphics[scale=0.425]{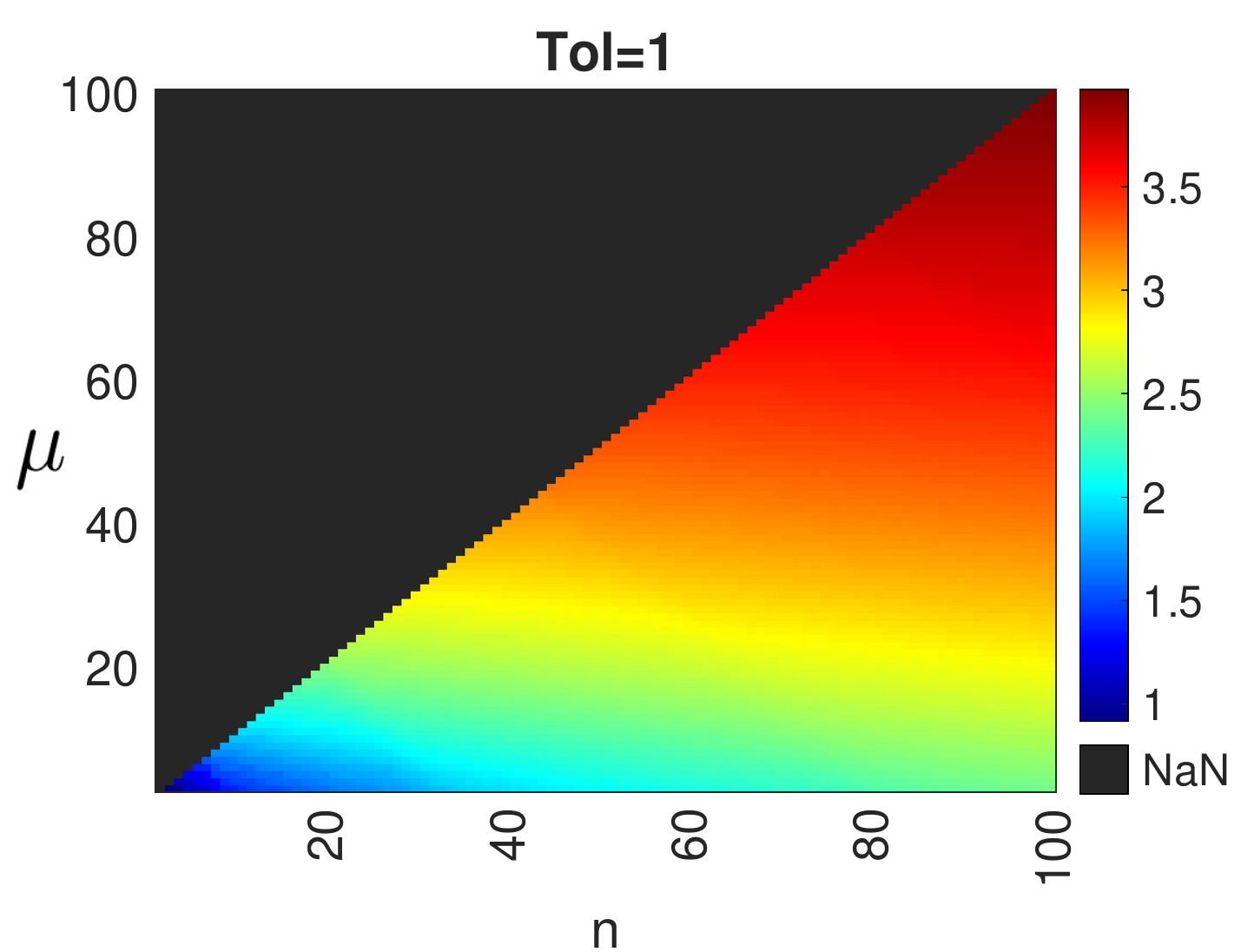}
\includegraphics[scale=0.425]{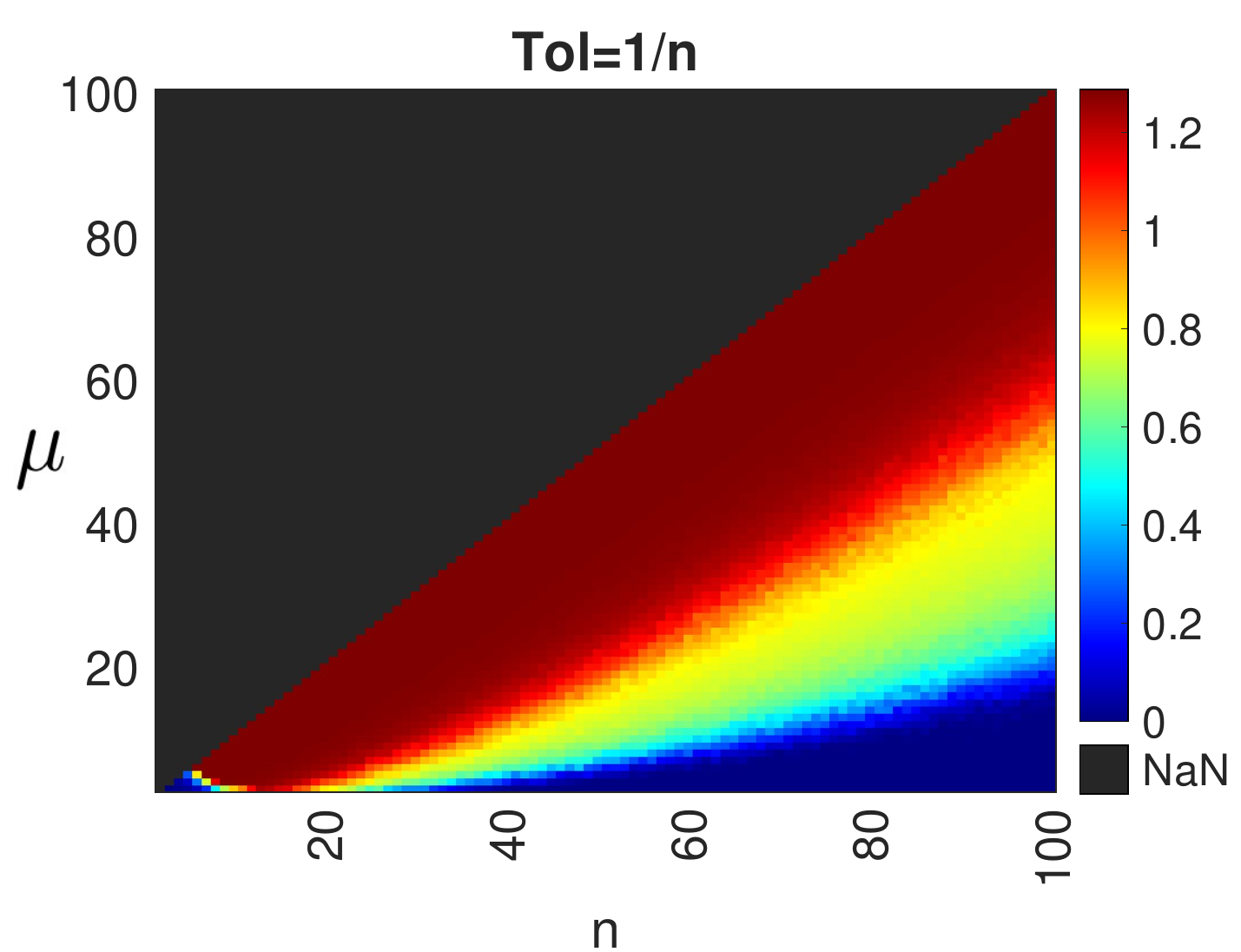}
\includegraphics[scale=0.425]{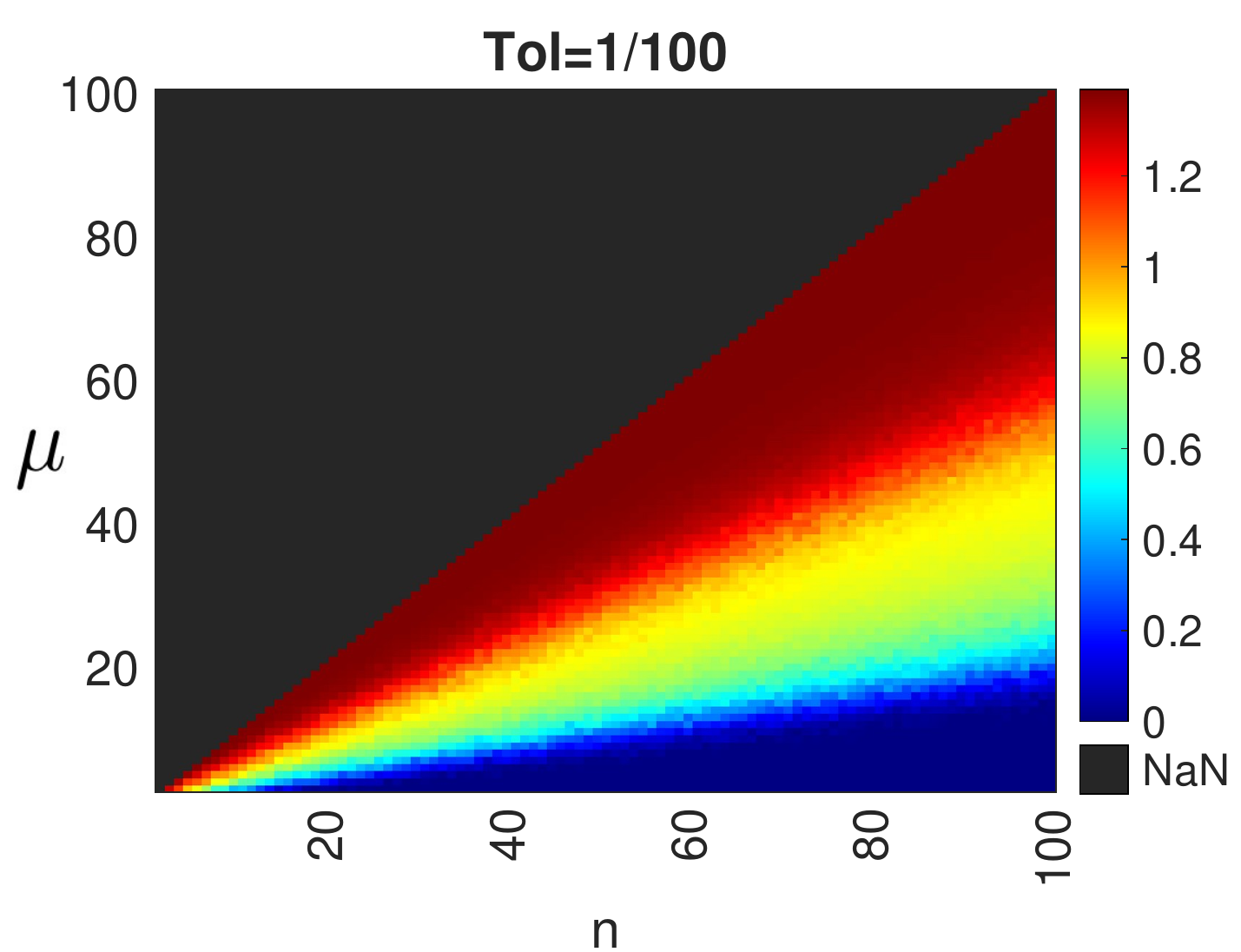}
\caption{Average results of the floating point-valued vector $v$ experiments. Each pixel at position $(n,\mu)$ represents the value of $\log(\alpha+t^*_{L_v}-t^*_{R_v})-\log(\alpha)$ averaged over $100$ subexperiments. The $x$-axis and $y$-axis correspond to $n=1,\ldots,100$ and $\mu=1,\ldots,n$ respectively. The tolerance used to define $t^*$ are absolute error less than: (a) $\textrm{exp}(1)$ (b) $1$ (c) $1/n$ (d) $1/100$. Black pixels above the diagonal are when $\mu > n$.}
\label{fig:experiment_uniform}
\end{figure}

\subsection{Uniformly distributed floating point numbers}
We repeat the experiments of the above section with floating point vectors $v \in [0,1]^n$ with $M=1$. Figure~\ref{fig:experiment_uniform} displays the results of comparing $L_v(t)$ and $R_v(t)$ on vectors whose elements are sampled from the uniform distribution on $[0,1]$. As with the previous experiments, each pixel at coordinates $(n,\mu)$ represents the value of $\log(\alpha+t^*_{L_v}-t^*_{R_v})-\log(\alpha)$, where $\mydef{t^*}$ is the first value of $t$ to approximate $M$ up to a given absolute error. To model a significant $g_2$ gap, we constructed these vectors by choosing $n-\mu$ vectors uniformly from $[0,1]$, multiplying them by $(n-1)/n$ and appending them to a vector of length $\mu$ with coordinates all equal to $1$. The figures differ only in the absolute tolerance used to define $t^*$. The value $\alpha = \min(t^*_{L_v}-t^*_{R_v})+1$ offsets the results so that when the results are averaged and plotted logarithmically, the minimum difference remains $\log(1) = 0$. The subtraction of the $\log(\alpha)$ term then better illustrates the subexperiments where $L_v(t)$ outperforms $R_v(t)$. This adjustment is accounted for in the uniform distribution examples as there are select instances in which $L_v(t)$ performs better than $R_v(t)$ by converging at a lesser $t$ value, thus the difference is non-positive and less than $-1$. This occurs most notably in the experiments with tolerance $1/n$ and $1/100$ of Figure~\ref{fig:experiment_uniform}.

\subsection{Clustering floating point numbers}
We repeat a similar experiment with floating point numbers in the presence of noise. Our goal is to identify the scenarios where it is appropriate to apply Theorem \ref{thm:Rtbounds} heuristically.
Our setup is as follows. Suppose that $5$ measurements, with values in $[0,1]$, are to be taken, but the measuring device incurs some error $\pm \epsilon$. To rectify this, each measurement is performed $20$ times. Heuristically, one may choose to apply Theorem \ref{thm:Rtbounds} with the interpretation that $v$ consists of $5$ numbers, each occurring with multiplicity $20$, with the goal of obtaining $\max(v)$ up to error $\epsilon$. In this case, Theorem \ref{thm:Rtbounds} specializes to
$t>\left(\frac{4\cdot g}{\epsilon }\right)^{\frac{1}{g}}$ where $g$ is the gap between the top two \emph{true} measurements.

After fixing $\epsilon$ and $g$, we model such a situation by the following procedure.
\begin{enumerate}
\item Pick $5$ true measurements $w_1,\ldots,w_5 \in [0,1]$ by setting $w_5=1$, $w_4=1-g$ and $w_1,w_2$ and $w_3$, sampled uniformly at random from $[0,1-g]$.
\item For each $w_i$, sample $20$ numbers uniformly from $[w_i-\epsilon, w_i+\epsilon]$. Collect all $100$ numbers in $v$.
\item Evaluate $R_v(t^*)$ for $t^* = \left( \frac{4g}{\epsilon}\right)^{\frac 1 g}$ to obtain the absolute error $\textrm{err}_{v}=|1-R(t^*)|$
\end{enumerate}
For each pair $(g,\epsilon)$, where $g=0.01,\ldots,1$ and $\epsilon=0,\ldots,1$, we repeat the above procedure $500$ times and average the error obtained in step $3$. Additionally, we deem an approximation a \emph{success} if the error is smaller than $\epsilon$. The two figures in Figure \ref{fig:clusterheatmap} display, for each pair $(g,\epsilon)$, the average error and number of successes.

\begin{figure}[htpb]
\includegraphics[scale=0.425]{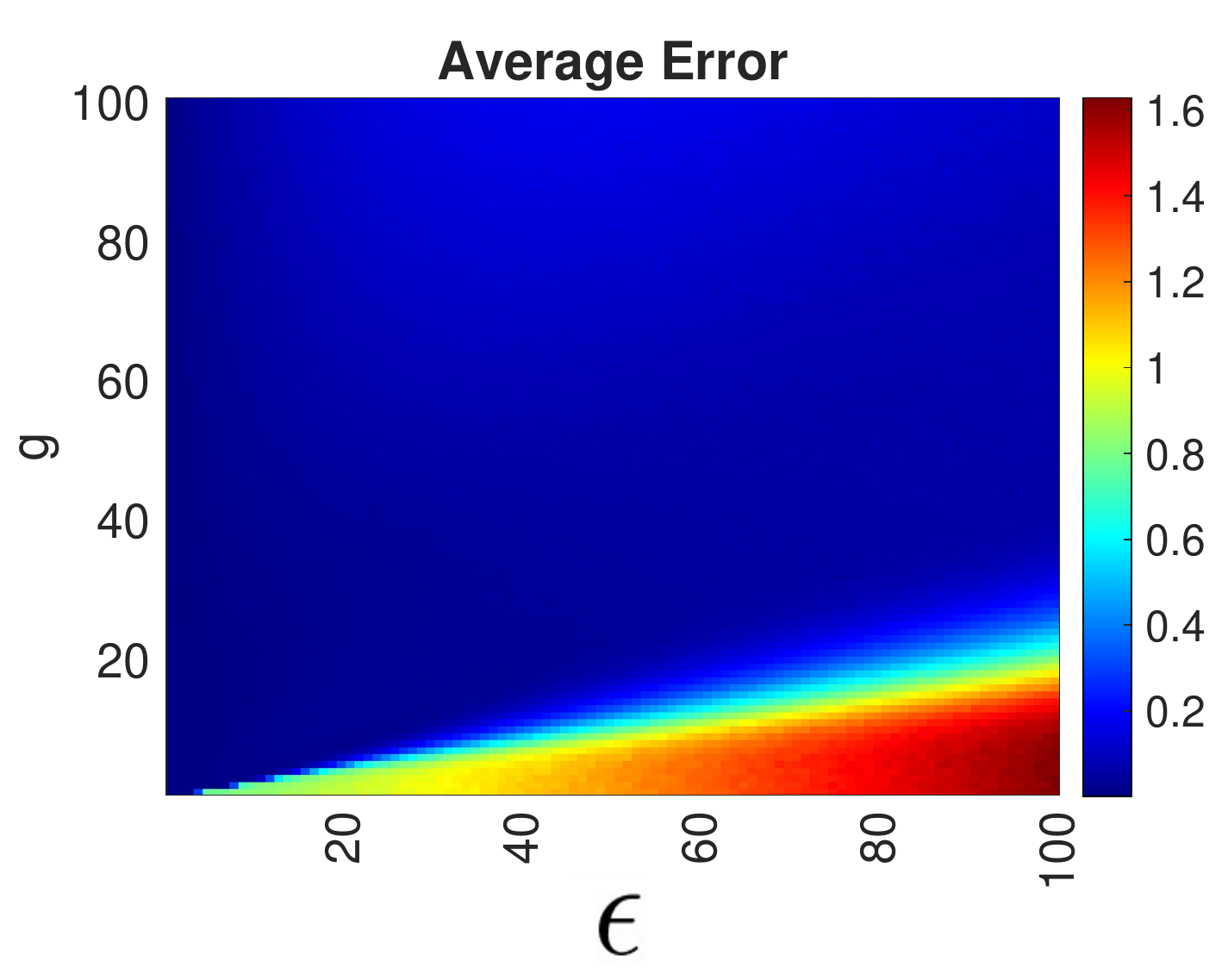}
\includegraphics[scale=0.425]{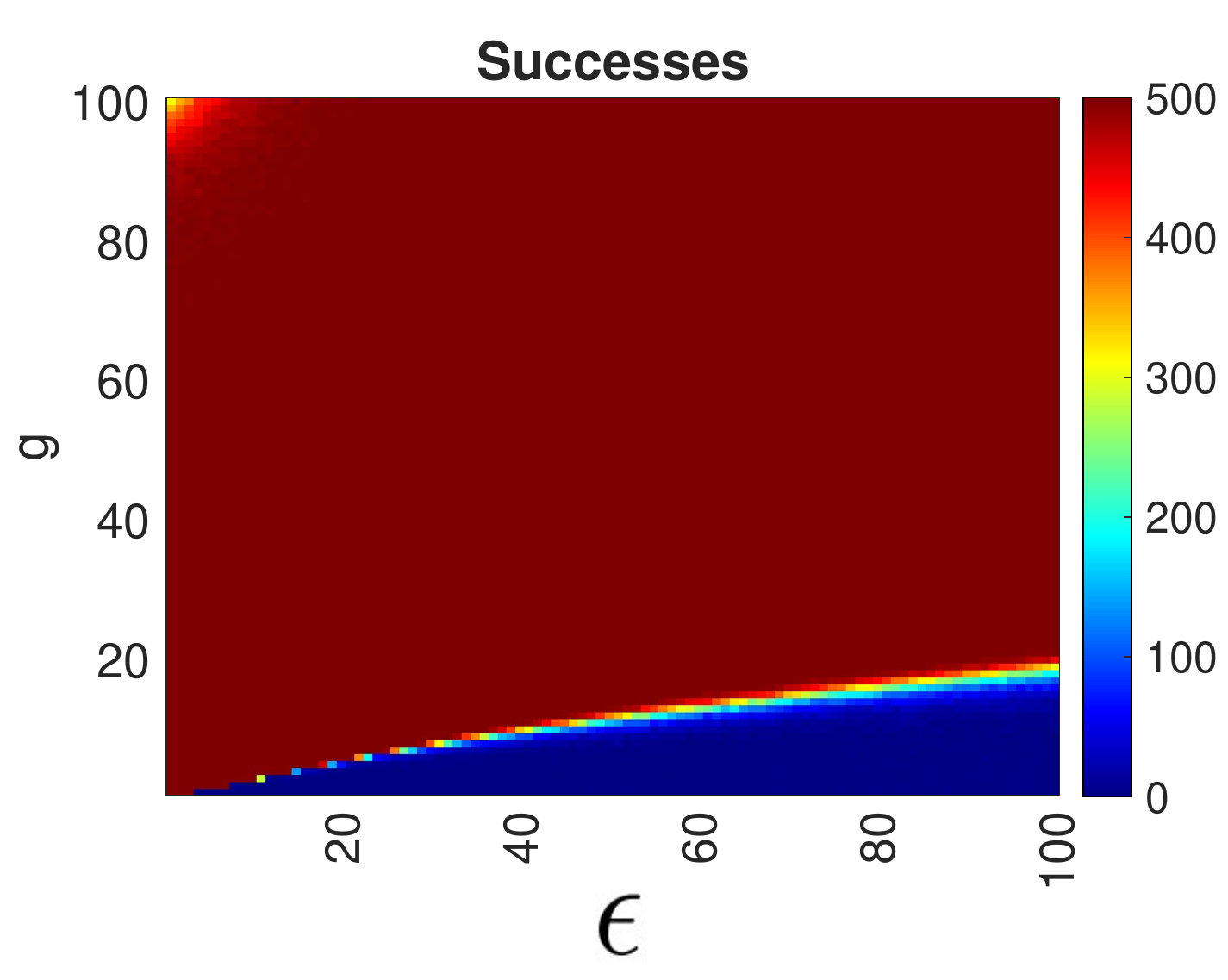}
\caption{For each $(g,\epsilon)$ and for $500$ tries, (left) average error of $R_v(t^*)$ for $t^*$ from Theorem~\ref{thm:Rtbounds} (right) number of approximations $R_v(t^*)$ within $\epsilon$ for $t^*$ from Theorem \ref{thm:Rtbounds}.}
\label{fig:clusterheatmap}
\end{figure}

As indicated by the experiments summarized in Figure \ref{fig:clusterheatmap}, a small gap $g$ and large $\epsilon$ produces the largest errors with the fewest numbers of successes, as expected. Interestingly, a large gap and small $\epsilon$, corresponding to the upper left corner of the figures also impacts the effectiveness of the heuristic. The large gap size means that $w_2,\ldots,w_5$ are all chosen within a small interval $[0,1-g]$, and we suspect that this cluster behaves like the value $\frac{1-g}{2}$ appearing with  high multiplicity. Additionally, the small $\epsilon$ value means it is difficult to achieve a ``success'' by having absolute error less than $\epsilon$.

For these noisy experiments, there are two natural interpretations of what should be considered $N$ and~$\mu$. The first is that $N$ should be $5$, the number of true measurements, whereas $\mu$ should be $1$. The second interpretation is that $N$ should be $20\cdot 5 = 100$, the true length of $v$ while $\mu$ should be $20$, the size of the top cluster. In the $R_v(t)$ case, these distinctions cancel out in the bound provided by Theorem \ref{thm:Rtbounds}. In the~$L_v(t)$ case, however, these interpretations give drastically different bounds when applying Theorem~\ref{thm:Ltbounds}. The later interpretation often yields such enormous $t$ bounds that an application of
that result is not useful.
In Figure~\ref{fig:LtCluster}, we display the results of an experiment using the former interpretation.
It did not happen that the $L_v(t)$ approximation with the interpreted bound
from Theorem~\ref{thm:Ltbounds}
achieved the expected accuracy.
This suggests that the later interpretation, despite its lack of utility, is likely more appropriate. Our experiments also showcase the advantage of using the $R_v(t)$ approximation over $L_v(t)$, especially in noisy situations with high (approximate) multiplicity.

\begin{figure}
\includegraphics[scale=0.5]{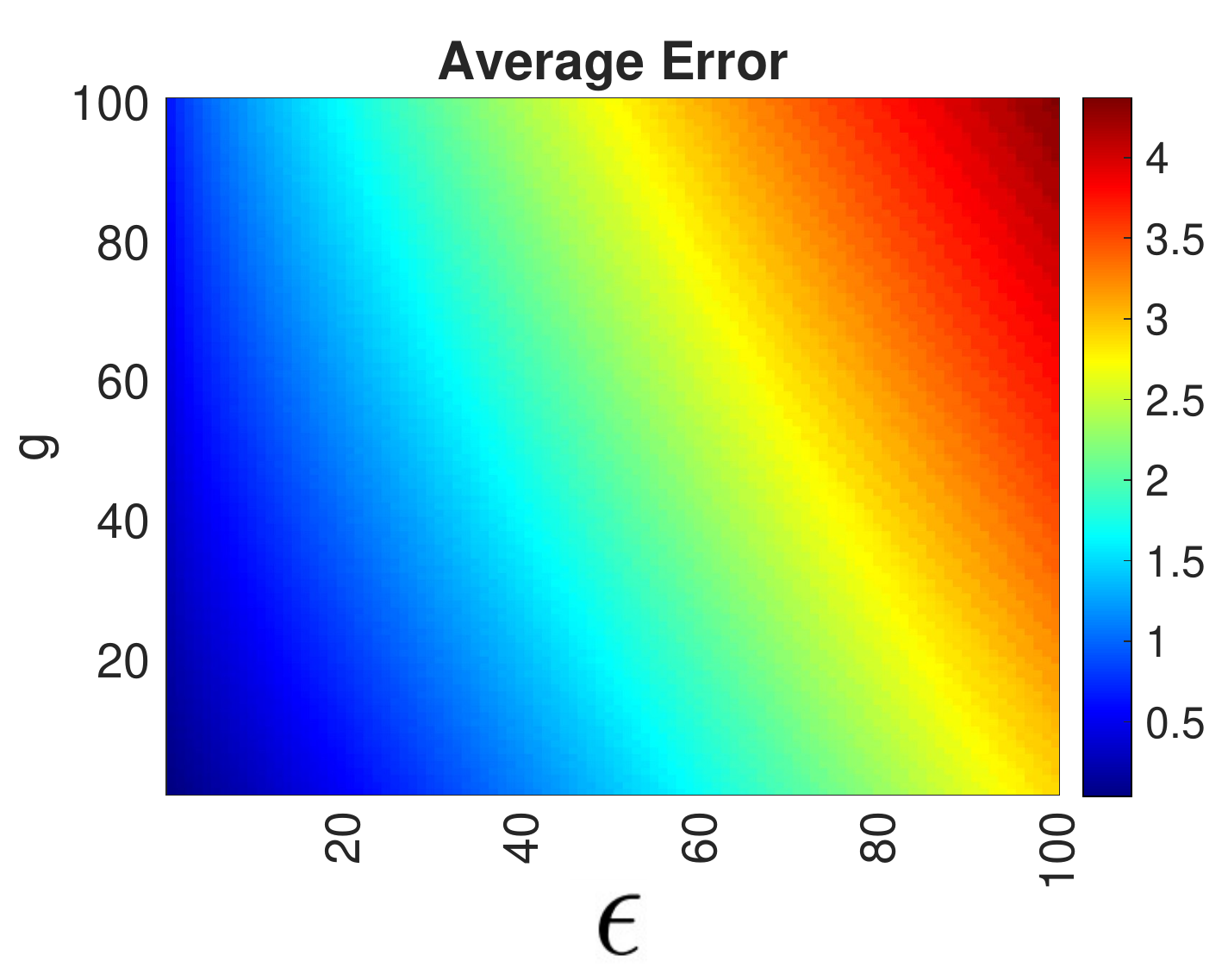}
\caption{For each $(g,\epsilon)$, over $500$ tries, the average error of $L_v(t^*)$ from Theorem~\ref{thm:Ltbounds}.}
\label{fig:LtCluster}
\end{figure}

\section{Max-convolution and applications}
\label{sec:MaxCon}

One way to use smooth approximations of
the maximum function is to approximate
the max-convolution of two vectors~\cite{Serang2015}.
To that end, consider
two integer vectors $a=(a_0,\ldots,a_n)$ and $ b=(b_0,\ldots,b_n)$.
The classical convolution problem asks to determine the vector of \mydef{convolution coefficients} $\mydef{a \star b}$ where \begin{equation}
\label{eq:convolutioncoefficients}(a \star b)_k=\sum_{i=0}^k a_i \cdot b_{k-i}=\sum_{i=\max(0,k-n)}^{\min(k,n)}a_i\cdot b_{k-i}.
\end{equation} We remark that the middle description is sufficient if one takes $a_i=0$ and $b_i=0$ whenever they are undefined. With the same input, the problem of max-convolution, \textbf{MAXCON}, asks for the vector $c$ of \mydef{max-convolution coefficients}, where
\begin{equation}
\label{eq:maxconvolutioncoefficients}\mydef{c_k} = \max_{\max(0,k-n) \leq i \leq \min(k,n)} (a_i+b_{k-i}).
\end{equation}
These coefficients can be obtained via \eqref{eq:convolutioncoefficients} by replacing the operations $(\cdot,+)$ with $(+,\max)$, respectively. The form $c_k=\max_{0 \leq i \leq n}(a_i+b_{k-i})$ may be used if one replaces undefined $a_i$ and $b_i$ with $-\infty$. Equivalently, through constructing $$\mydef{A_t(x)}=\sum_{i=0}^n t^{a_i}x^i, \quad \quad \mydef{B_t(x)} = \sum_{i=0}^n t^{b_i}x^i \in \mathbb{Q}(t)[x],$$ the problem of \textbf{MAXCON} asks for the largest exponents in $t$ appearing in the coefficients of $$A_t(x) \cdot B_t(x) = \sum_{k=0}^{2n} \sum_{i=0}^k t^{a_i+b_{k-i}}x^{k}.$$
Setting $\mydef{v^{(k)}}=((a_i+b_{k-i})~|~ \max(0,k-n)\leq i \leq \min(k,n))$,   we rewrite this as
$$A_t(x)\cdot B_t(x) = \sum_{k=0}^{2n} F_{v^{(k)}}(t)x^k.$$
{For fixed $t$, the values of $F_{v^{(k)}}(t)$ are \emph{classical} convolution coefficients $$((t^{a_0},\ldots,t^{a_n}) \star (t^{b_0},\ldots,t^{b_n}))_k$$ which can be computed with $O(n\log(n))$ operations, e.g., using the fast Fourier transform (FFT) \cite{FastMultiply}. Applying $\log_t$ provides an $O(n\log(n))$ routine for evaluating $L_{v^{(k)}}(t)$, whereby
with Theorem~\ref{thm:Ltbounds}, this process computes $\max(v^{(k)})=c_k$ when evaluated at a sufficiently large value of $t$. We summarize this discussion in the following
quasi-linear time algorithm.}

\begin{algorithm}[!htpb]
\SetKwIF{If}{ElseIf}{Else}{if}{then}{elif}{else}{}%
\DontPrintSemicolon
\SetKwProg{MaxCon}{MaxCon}{}{}
\LinesNotNumbered
\KwIn{
 Two integer vectors $a=(a_0,\ldots,a_n)$ and $b=(b_0,\ldots,b_m)$.
}
\KwOut{The max-convolution coefficients $c=(c_0,\ldots,c_{n+m})$}
\nl Choose $t^*$ satisfying the bounds of Theorem \ref{thm:Ltbounds} (e.g., $t^*\geq\max(n,m)$+1) \;
\nl Compute $\alpha = ((t^{*})^{a_0},\ldots,(t^{*})^{a_n})$ and $\beta = ((t^{*})^{b_0},\ldots,(t^{*})^{b_m})$\;
 \nl Compute $\ell(t^*)=\alpha \star \beta$\;
 \nl Apply $\lfloor \log_{t^*}(\cdot) \rfloor$ component-wise to $\ell(t^*)$ to obtain $c=\lfloor\log_{t^*}(\ell(t^*))\rfloor$\;
 \nl \Return{} $c$\;
\caption{MaxCon \label{algo:MaxConL}}
\end{algorithm}

\begin{example}\label{ex:MaxCon1}
Consider applying
Algorithm \ref{algo:MaxConL} to compute the max-convolution coefficients of the vectors $$a=(3,1,2,4,1,2), \quad \quad b=(5,3,0,4).$$ Taking $t^*=6$, we obtain
$$\alpha = (216,6,36,1296,6,36), \quad \quad \beta = (7776,216,1,1296)$$
The classical convolution of $\alpha$ and $\beta$ is
{\footnotesize{$${\ell(6)=\alpha \star \beta = (1679620, 93312, 281448, 10365400, 334404, 329184, 1687400, 7812, 46656
)}
$$}}
which under $\log_6$ evaluates to
{\footnotesize{
 $$\log_6(\ell(6)) = (8, 6.38685, 7.00301, 9.01571, 7.09923, 7.09045, 8.00258, 5.00258, 6)
.$$}}
 Finally, by applying $\lfloor \cdot \rfloor$, we obtain the max-convolution coefficents
 $$c = (8,6,7,9,7,7,8,5,6).$$

For completeness and interpretation of $c$, we provide $A_t(x),B_t(x) \in \Z[t][x]$ below along with their product:
$$A_t(x) = t^3x^0+t^1x^1+t^2x^2+t^3x^3+t^1x^4+t^2x^5, \quad \quad B_t(x) = t^5x^0+t^3x^1+t^0x^2+t^4x^3,$$
\begin{align*}
A_t(x)\cdot B_t(x) =
t^8x^0+2t^6x^1+(t^7+t^4+t^3)x^2+(t^9+t^7+t^5+t)x^3+(t^7+t^6+t^5+t^2)&x^4 \\ +~(t^7+t^6+2t^4)x^5+(t^8+t^5+t)x^6+(t^5+t^2)x^7+t^6&x^8.\end{align*}
\end{example}

The key to Algorithm \ref{algo:MaxConL} lies in the ability to evaluate $F_{v^{(k)}}(t)$ via the fast Fourier transform in $O(n\log(n))$ by interpreting these values as classical convolution coefficients. A subsequent application of $\log_t$ turns this into an evaluation of $L_{v^{(k)}}(t)=\log_t(F_{v^{(k)}}(t))$ whereby one may apply Theorem \ref{thm:Ltbounds}. Similarly, with the aim to apply Theorem \ref{thm:Dtbounds}, one may use the approximation of Proposition \ref{prop:ApproxR} to wrap the ability to evaluate~$F_{v^{(k)}}(t)$ into an algorithm which may require a smaller $t$ evaluation.
\vspace{0.2 in}

\begin{algorithm}[H]
\SetKwIF{If}{ElseIf}{Else}{if}{then}{elif}{else}{}%
\DontPrintSemicolon
\SetKwProg{MaxCon}{MaxCon}{}{}
\LinesNotNumbered
\KwIn{
 Two integer vectors $a=(a_0,\ldots,a_n)$ and $b=(b_0,\ldots,b_m)$.
}
\KwOut{The max-convolution coefficients $c=(c_0,\ldots,c_{n+m})$}
\nl Choose $t^*$ and $\alpha^*$ satisfying the bounds of Theorem \ref{thm:Dtbounds} (e.g., $\alpha^* > 1$ and $t^*>\textrm{max}(e,n-1,m-1)$)  \;
\nl Compute $\alpha = ((t^{*})^{a_0},\ldots,(t^{*})^{a_n})$ and $\beta = ((t^{*})^{b_0},\ldots,(t^{*})^{b_m})$ \;
\nl Compute $\alpha' = ((\alpha^* t^{*})^{a_0},\ldots,(\alpha^* t^{*})^{a_n})$ and $\beta' = ((\alpha^* t^{*})^{b_0},\ldots,(\alpha^* t^{*})^{b_m})$\;
 \nl Compute $\ell(t^*)=\alpha \star \beta$ \;
 \nl Compute $\ell'(t^*)=\alpha' \star \beta'$ using FFT\;
 \nl Apply $\log_{\alpha^*}\left(\ell'(t^*)/\ell(t^*)\right)$ component-wise to obtain $c = \lceil \log_{\alpha^*}\left(\ell'(t^*)/\ell(t^*)\right)\rceil$\;
 \nl \Return{} $c$\;
\caption{MaxCon - using $D_{v^{(k)}}(t^*,\alpha^*)$ \label{algo:MaxConD}}
\end{algorithm}

\begin{remark}
Algorithms \ref{algo:MaxConL} and \ref{algo:MaxConD}, paired with their corresponding bounds from Section \ref{sec:Approximations}, give algorithms whose output constitute mathematical proofs provided that the convolution coefficients computed via FFT are exact. Otherwise, the error introduced by the $\star$ operation must be bounded by $1/2$, $\delta$ should be taken to be at most $1/2$ in the relevant theorems, and a two-sided rounding procedure should be applied rather than the ceiling or floor function.
\end{remark}

\begin{example}

We apply Algorithm \ref{algo:MaxConD} to the
vectors in Example~\ref{ex:MaxCon1}
using $t^*=6$ and $\alpha^* = e$.
The values of $\ell(6)$ and $\ell'(6) = \ell(e\cdot 6)$ are
{\tiny{
\begin{align*}
\ell(6) & = (1679620, 93312, 281448, 10365400, 334404, 329184, 1687400, 7812, 46656
),\\
\ell'(6)& = (  5006864730.36308,
    37644747.57839,
   307062197.51625,
81968557661.46344,\\
  & \quad \quad   326963800.35823,
   325950992.03191,
  5008018807.39795,
     1154326.73120,
    18822373.78920)
\end{align*}
}}
so that $c=\lceil \log(\ell'(6)/\ell(6))\rceil$ gives
$$c = \lceil (8.0,
 6.0,
 6.99486,
 8.97562,
 6.88525,
 6.89789,
 7.99561,
 4.99561,
 6.0)\rceil =
 (8, 6, 7, 9, 7, 7, 8, 5, 6).
 $$
We remark that if we use $t^*=2$ and $\alpha^*= 1.05$, which do not necessarily meet the bounds of Theorem~\ref{thm:Dtbounds}, we
obtain the following results:
{\small{
\begin{align*}
\ell(2) &= (256,
 128,
 152,
 674,
 228,
 224,
 290,
  36,
  64), \\
  \ell'(2) &= \ell(1.05 \cdot 2) =   (378.22859,
  171.53224,
  208.81795,
 1017.32991,
  311.12599,\\
  & \quad \quad  \quad \quad \quad \quad  \quad \,\,
  304.77118,
  421.16960,
   45.25101,
   85.76612), \\
   \log_{1.05}(\ell'(2)&/\ell(2)) = ( 8.0,
 6.0,
 6.50915,
 8.43831,
 6.37121,
 6.31101,
 7.64815,
 4.68754,
 6.0).
\end{align*}
}}
A final application of $\lceil \cdot \rceil$ obtains the correct integers $c$ for the max-convolution coefficients of $a$ and $b$.
\end{example}

We conclude with two applications
of max-convolution.

\subsection{Maximum Consecutive Subsums Problem} Given a single vector $v = (v_1,\ldots,v_n)\in\mathbb{R}^n$, the problem of determining the largest consecutive sum $\sum_{i=1}^k v_{j_k+i}$
for each $k=1,\dots,n$
is known as the \emph{Maximum Consecutive Subsums Problem} (\textbf{MCSP}).
As outlined in \cite[\S~7.1]{CMWW}, \textbf{MCSP} directly reduces to an instance of \textbf{MAXCON} as follows. Taking $a,b \in \mathbb{R}^n$ to be $a_k = -\sum_{i=1}^k v_i$ and $b_{n-k} = \sum_{i=1}^k v_i$, the max-convolution coefficient $c_{n-k}$ describes the largest sum of $k$ consecutive entries of $v$.
\begin{example}
For $v=(1, 4, 2, 3, 8, 1, 1, 5, 6, 7, 5)\in\mathbb{Z}^{11}$, we have
\begin{align*}a&=(-1, -5, -7, -10, -18, -19, -20, -25, -31, -38, -43), \\ b&=(43, 38, 31, 25, 20, 19, 18, 10, 7, 5, 1).
\end{align*}
The max-convolution of $a$ and $b$ is
$$c = (42, 38, 36, 33, 28, 24, 23, 18, 13, 8, 0, -1, -2, \ldots).$$
For example, this shows that the largest sum
 of $2$ and $5$ consecutive entries of $v$ is $13$ and $24$ obtained by $6+7$ and $1+5+6+7+5$, respectively.
By convention, one may choose to prepend $c_0=\sum_{i=1}^n v_i$ to $c$ so as to include the subsum of $n$ consecutive integers in the output as well.
\end{example}

We remark that even though Algorithms \ref{algo:MaxConL} and \ref{algo:MaxConD} are written for integer input, the algorithms work for floating point input as well, subject to different bounds (see Theorems \ref{thm:Ltbounds} and \ref{thm:Dtbounds}). When using Algorithm~\ref{algo:MaxConL}, the output is an upper bound for the true max-convolution coefficients, subject to any error introduced by the FFT subroutine. Figure \ref{fig:MCSP} shows the magnitude of error on the $100$ outputs of a random MCSP problem on a vector $v\in[0,1]^{n}$
for $n=100$ and $n=1000$
with each coordinate selected uniformly
at random.

\begin{figure}
\includegraphics[scale=0.45]{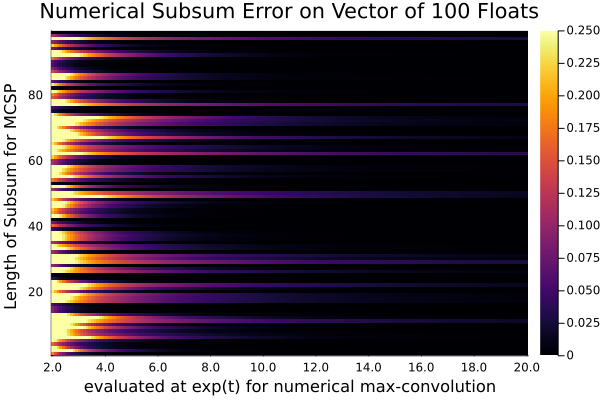}
\includegraphics[scale=0.45]{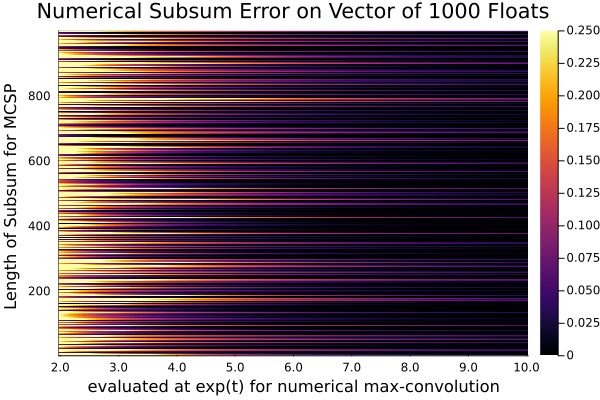}
\caption{Heatmap of the absolute errors of the numerical computation of \textbf{MCSP}. The numerical computations were performed on a vector of length $100$ (top) and $1000$ (bottom) with floating-point entries uniformly chosen in $[0,1]$. The vertical axis indicates the output error on subsums of length $k$ after evaluating at $\textrm{exp}(t)$~(horizontal~axis).
}
\label{fig:MCSP}
\end{figure}

\subsection{Service Curve Constraints}
Convolution algorithms are integral in network calculus where systems model the data flow between networks~\cite{NetworkCalc}. The incoming data is described by a monotonic input function~$R(T)$, given in bits per second. The outgoing data (after a time delay) is described by the output function,~$R^*(T)$, also in bits per second.
The function $R^*(T)$ is constrained by \textit{service} constraints that state for any window of time, additional data outputted is bounded.
The curves formed by these constraints  are the result of a min-convolution between the service curve and input function $R(T)$~\cite{NetworkCalc}.
That is, a system with an input function $R(T)$ has an output function $R^*(T)$ that will lie in the area bounded below by a service curve $\beta(T)$ and above by a maximum service curve $\gamma(T)$ such that
\begin{equation}
 \inf_{s\leq T}\{ R(T) + \beta(T-s)\} \leq R^*(T) \leq R(s) + \gamma(T-s), \hspace{5mm} s\leq T
\label{eq:serviceConstr}
\end{equation}
Note that the input and output functions admit no subscript to avoid confusion with the ratio function~$R_v(t)$. Additionally, we define the time variable for the service curve to be $T$ rather than $t$ which is the variable base used for the \textbf{MAXCON} algorithms.

Although the prior focus was on
the maximum, it is very simple to
reformulate everything to instead compute the minimum. That is,
we take ${t\rightarrow \infty}$
when converging to the maximum
while one can take ${t\rightarrow 0^+}$ to converge to the minimum.

As an example, we sought to recreate~\cite[Fig.~5.1]{NetworkCalc} using Algorithm~\ref{algo:MaxConD} to compute the discrete min-convolution of the input function and service curves.
To complete this recreation, we first
fit a polynomial curve to $R(T)$
from that plot.  In particular,
we used the fitted sextic polynomial
\[
R(T) = 1.6738T-0.7492T^2-0.08694T^3+0.1085T^4
-0.01101T^5-0.001579T^6+0.0002085T^7.
\]
The service curves are defined as $\beta(T) = T$ and $\gamma(T) = \begin{dcases}
0 & \hbox{if~} T \leq 3\\
T-3 & \hbox{if~} T > 3,
\end{dcases}$
which corresponds with a 3 second time delay.
To create a discrete problem,
we evaluated these functions at equally spaced points.

\begin{figure}[htb]
\includegraphics[scale=0.63]{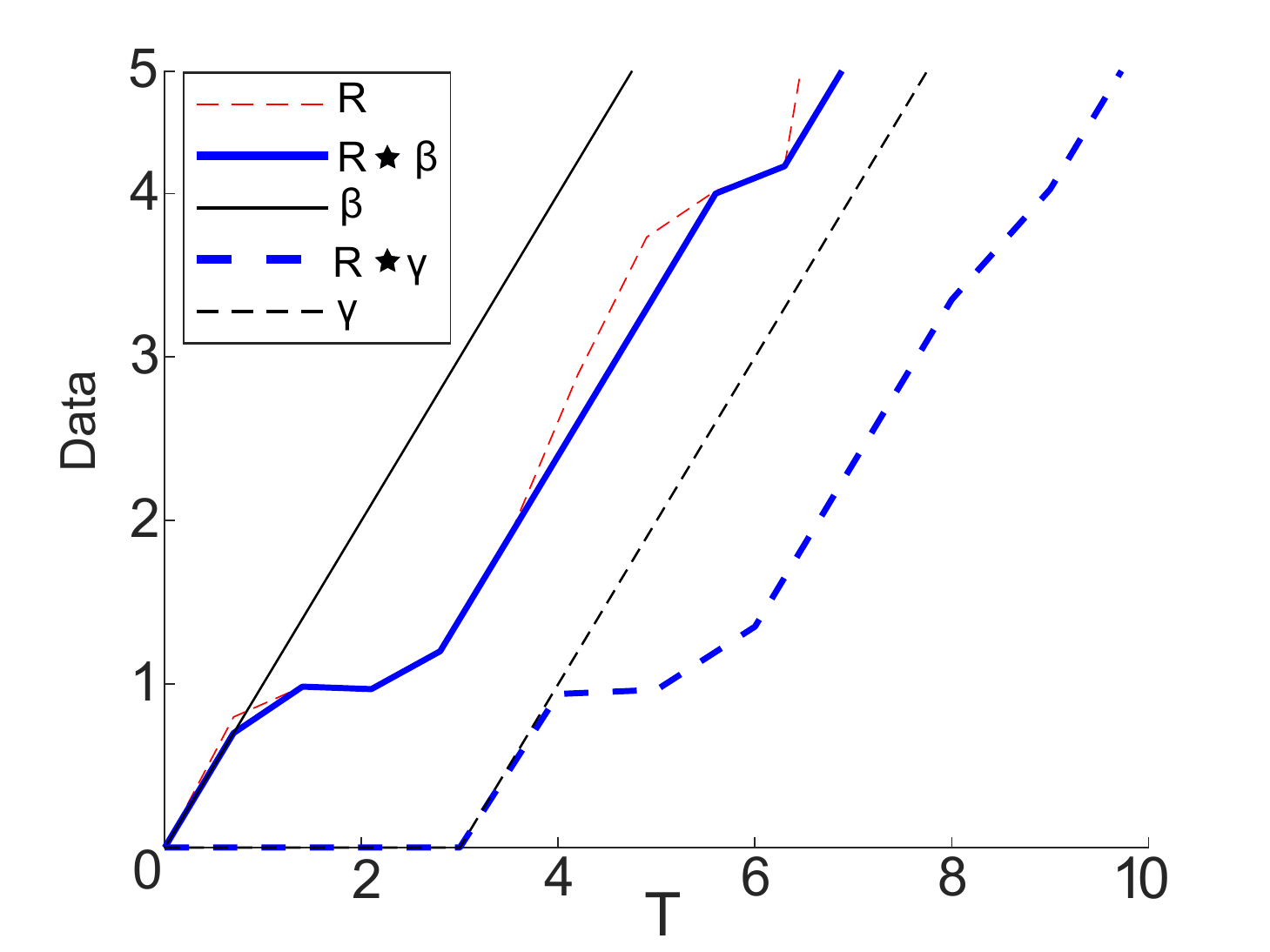}
\includegraphics[scale=0.45]{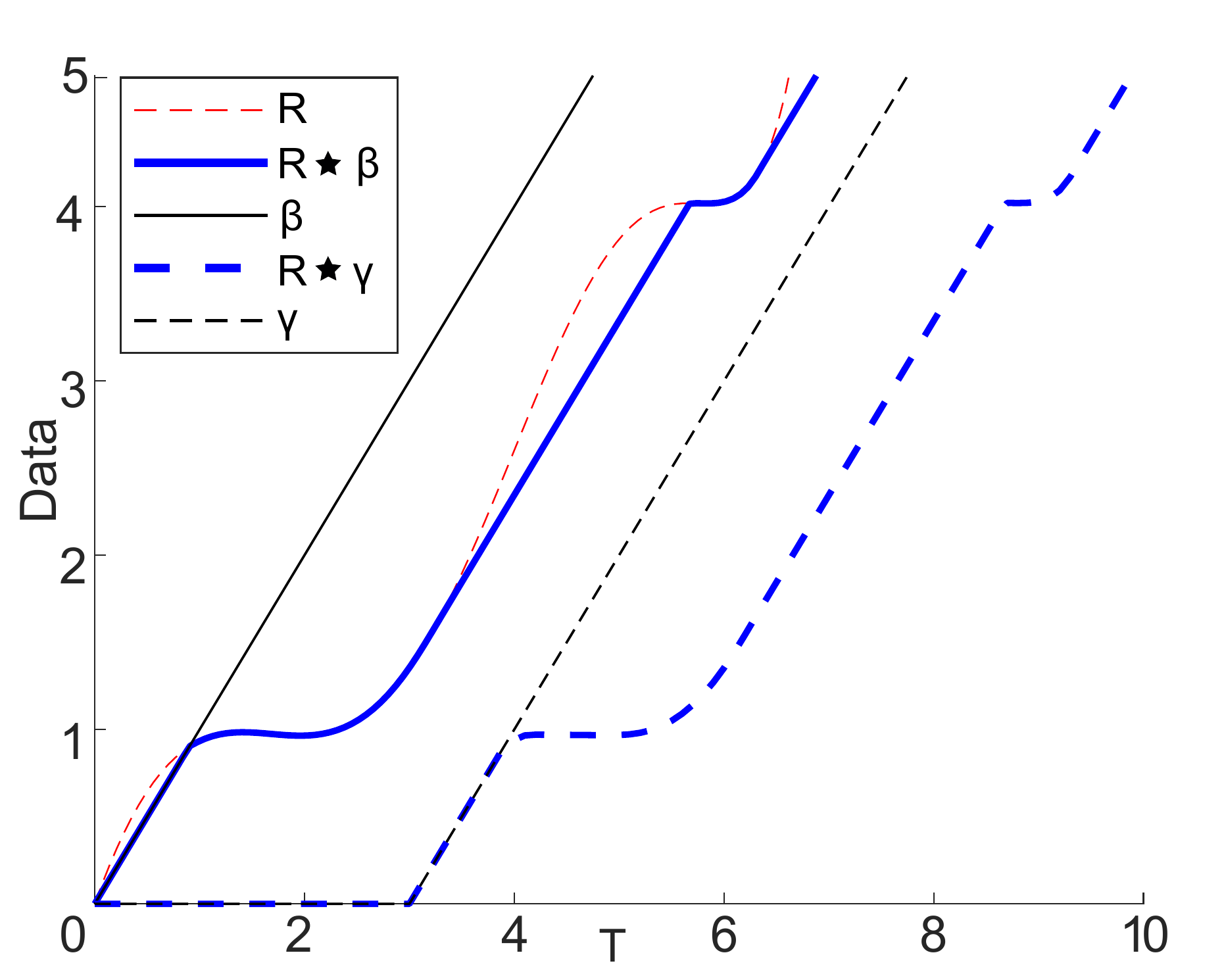}
\caption{Discretized service curves computed using Algorithm~\ref{algo:MaxConD}.
The curve $R(T)$ and service constraints $\gamma(T)$ and $\beta(T)$ are discretized into 10 (top) and 100 (bottom) equally spaced points.}
\label{fig:ServiceCurveEx}
\end{figure}

We apply Algorithm~\ref{algo:MaxConD} in the floating-point case, i.e., without rounding, for the computations
with $\alpha = 1.01$ and $t = \left(\frac{1}{(n-1)}\right)^{\frac{1}{0.04}}$.
Figure~\ref{fig:ServiceCurveEx} shows the results of our min-convolution using 10 and 100 discretized points
to compute the corresponding 
bounds on $R^*(T)$ given in~\eqref{eq:serviceConstr}.

\begin{figure}[htpb]
\includegraphics[width=0.49\textwidth]{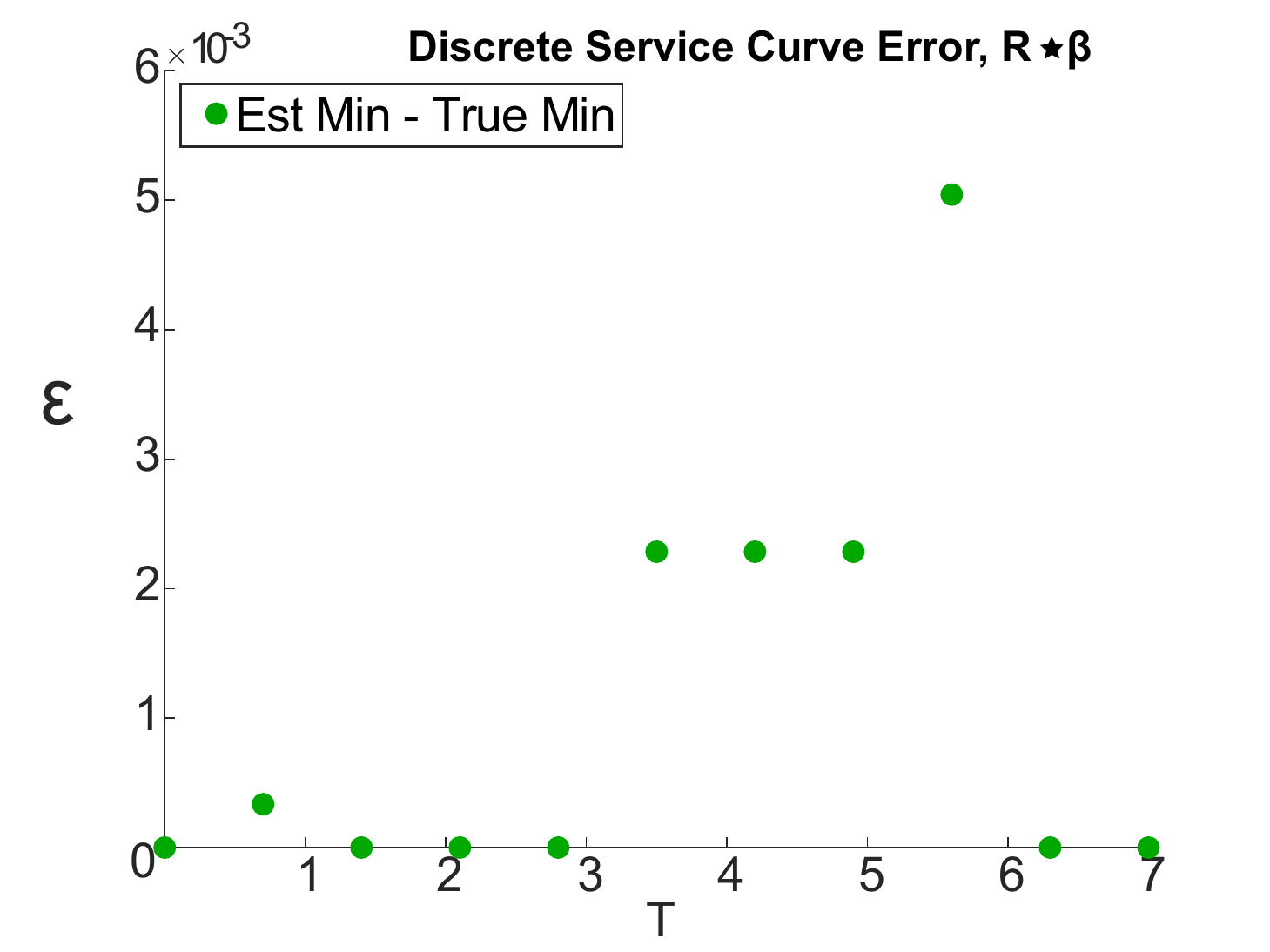}
\includegraphics[width=0.49\textwidth]{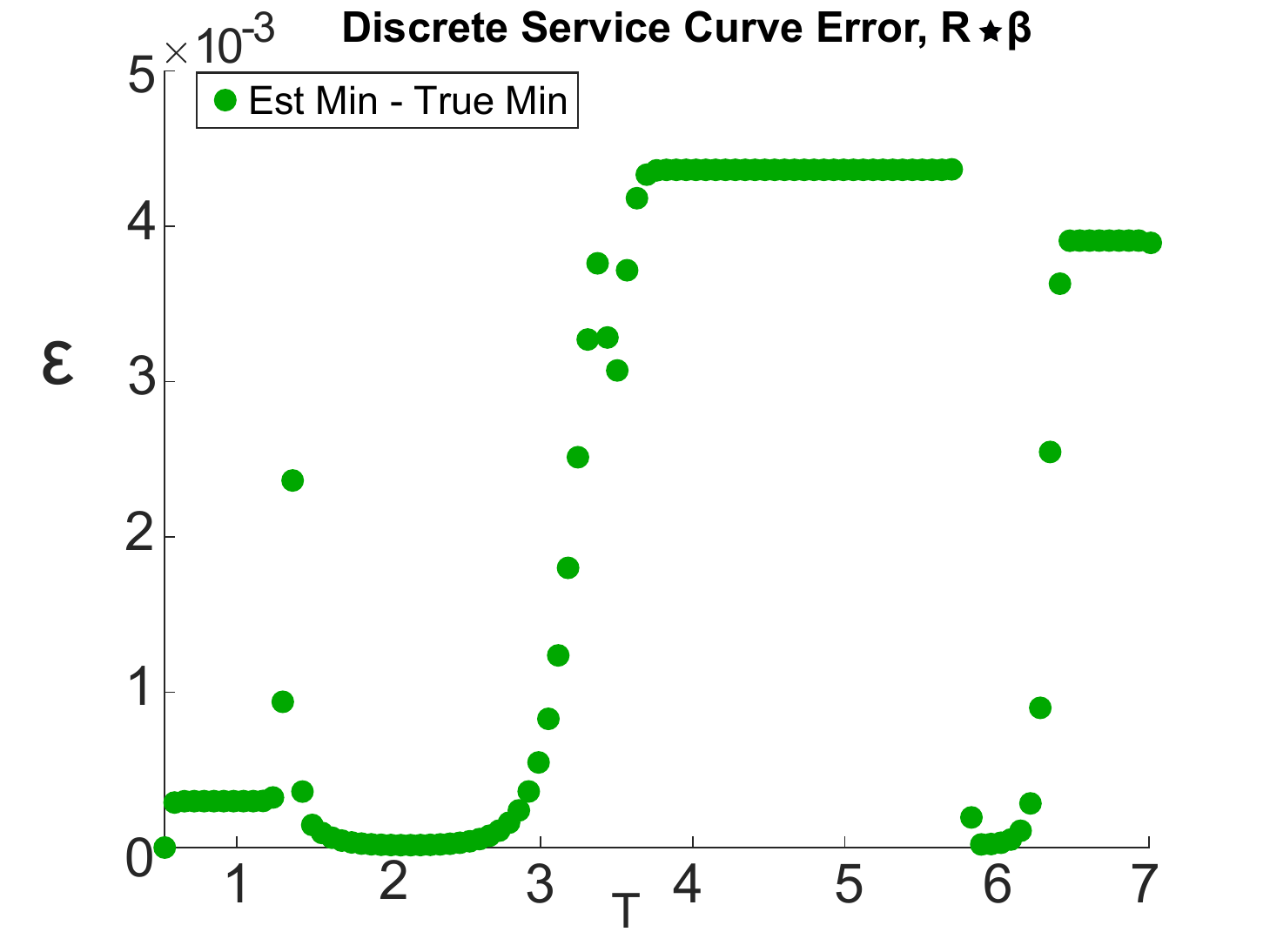}
\caption{Plots error of the computed points minus the actual points of the minimum convolution between the lines $R(T)$ and $\beta(T)$ for (left) 10 and (right) 100 discretized points.}
\label{fig:BetaErr}
\end{figure}

\begin{figure}[htbp]
\includegraphics[width=0.49\textwidth]{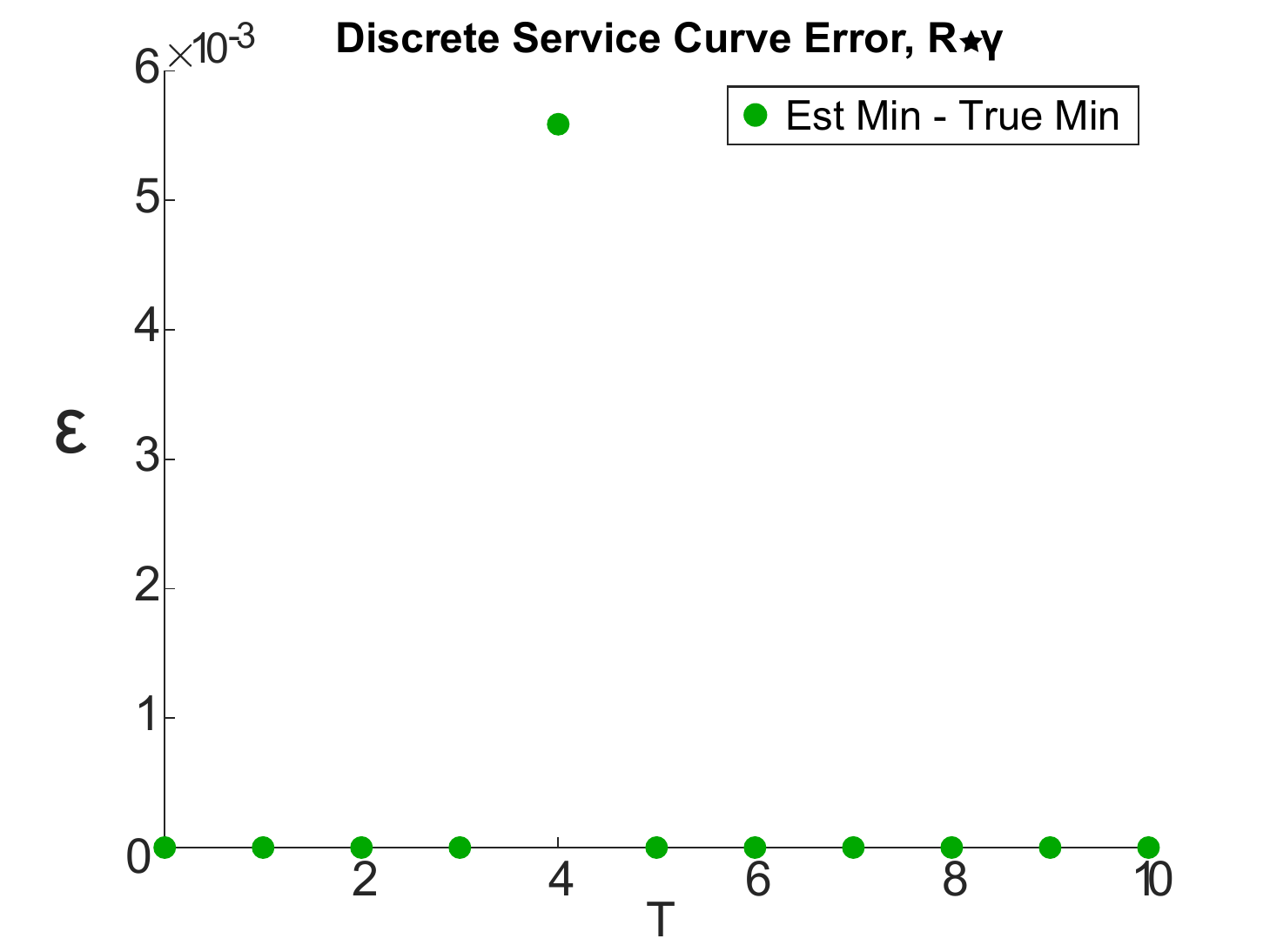}
\includegraphics[width=0.49\textwidth]{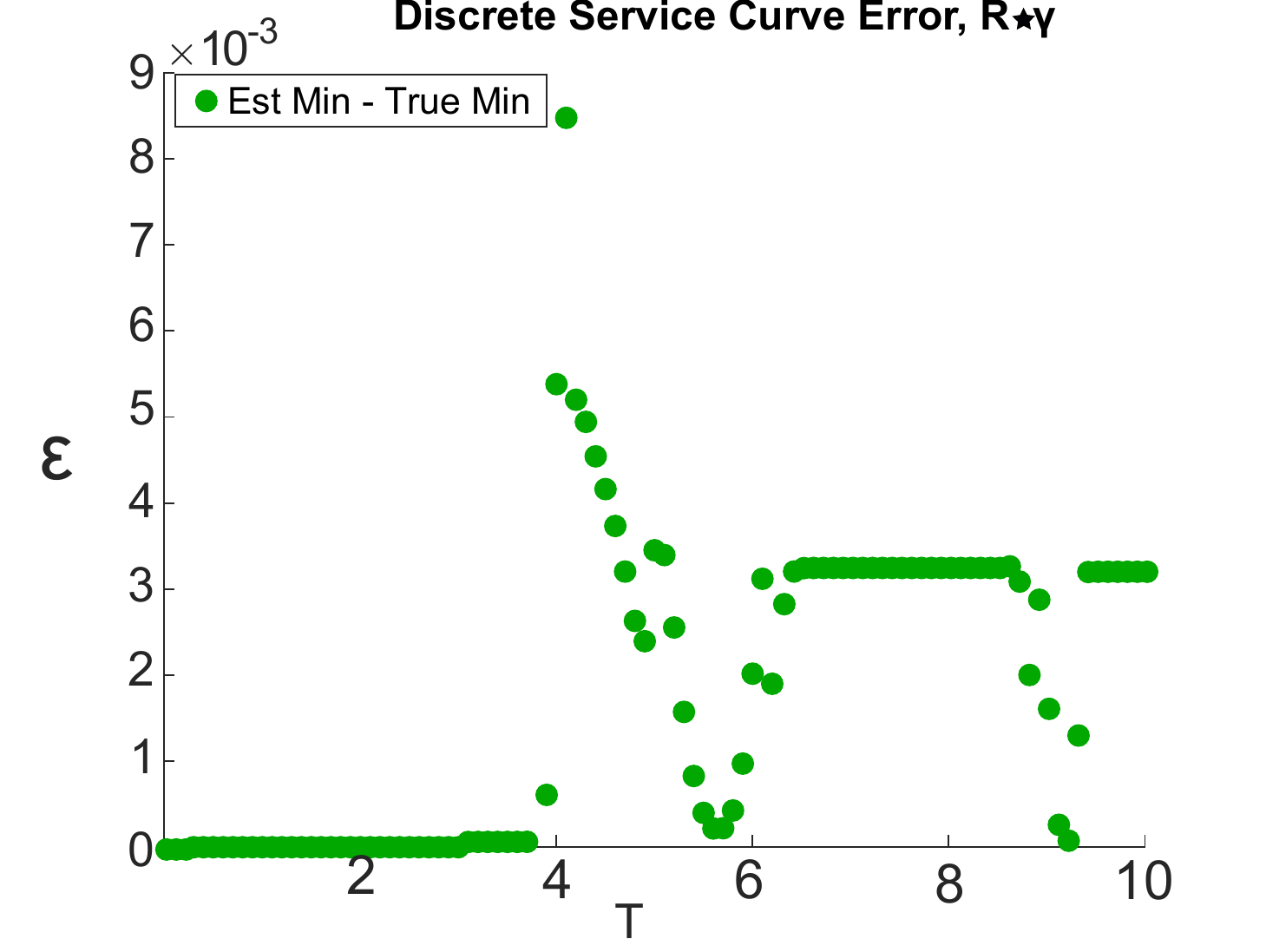}
\caption{Plots error of the computed points minus the actual points of the minimum convolution between the lines $R(T)$ and $\gamma(T)$ for (left) 10 and (right) 100 discretized points.}
\label{fig:GammaErr}
\end{figure}

One can see that even with numerical discretion, the resulting curves in Figure~\ref{fig:ServiceCurveEx} exhibit satisfactory behavior in recreating the bounds of \cite[Fig.~5.1]{NetworkCalc}. Even for 10 discretization points, the computation captures the essential behavior of the convolution.  Furthermore,
using 100 discretization points 
better captures sharp transitions
as well as flatter regions 
of the service curve bounds.

Note that this experiment employs two separate convolutions. We compared our estimated values to the actual minimums computed via brute force, i.e., 
we computed the exact minima 
by computing the bounds given in Equation~\ref{eq:serviceConstr} for each $T_i$, $i = 1, \dots, N$ where N is the number of discretized points. Figures~\ref{fig:BetaErr}
and~\ref{fig:GammaErr} display the error between the computed points via Algorithm~\ref{algo:MaxConD} 
and the actual points for the min-convolution between $R(T)$ and $\beta(T)$,
and $R(T)$ and $\gamma(T)$, respectively.
When using a time delay of $3$ seconds,
non-zero values first occur when $T>3$
resulting in nearly zero error before then.
Both discretizations have errors 
on the order of $10^{-3}$ or smaller. Additionally, the errors are nonnegative which highlights that our method slightly overestimates the values.

\begin{table}[!htbp]
\begin{center}
\begin{tabular}{|d{5.2}|d{4.7}|d{4.7}|}
\hline
\multicolumn{1}{|c|}{Number of Discrete Points} & \multicolumn{1}{|c|}{$D_{v}(t,\alpha)$)} & \multicolumn{1}{|c|}{Brute Force}\\
\hline
10 & 0.0003841 &  0.002666\\
50 & 0.0003392 & 0.006488\\
100 & 0.0005358 & 0.01067\\
500 & 0.001053 & 0.04643\\
1000 & 0.001647 & 0.1114\\
5000 & 0.01183 &  1.3957\\
10000 &  0.02939 &  10.5240\\
100000 & 1.1738 & 1546.2364\\
\hline
\end{tabular}
\caption{Time, in seconds, to calculate the min-convolution between $R(T)$ and $\beta(T)$.}
\label{tb:convTimesBeta}
\end{center}
\end{table}

\begin{table}[!htbp]
\begin{center}
\begin{tabular}{|d{5.2}|d{4.7}|d{4.7}|}
\hline
\multicolumn{1}{|c|}{Number of Discrete Points} & \multicolumn{1}{|c|}{$D_{v}(t,\alpha)$} & \multicolumn{1}{|c|}{Brute Force}\\
\hline
10 & 0.0001468 &  0.002490\\
50 & 0.0001085 & 0.005987\\
100 & 0.0001927 & 0.007235\\
500 & 0.0007843 & 0.03623\\
1000 & 0.001421 & 0.1315\\
5000 & 0.006734 &   1.5642\\
10000 &  0.01953 &  5.2610\\
100000 &  0.7965 & 1239.4949\\
\hline
\end{tabular}
\caption{Time, in seconds, to calculate the min-convolution between $R(T)$ and $\gamma(T)$.}\label{tb:convTimesGamma}
\end{center}
\end{table}

Tables~\ref{tb:convTimesBeta} 
and~\ref{tb:convTimesGamma}
compare the time in seconds for one application of Algorithm~\ref{algo:MaxConD} and a brute force computation between $R(T)$ and $\beta(T)$,
and $R(T)$ and $\gamma(T)$, respectively,
using a single processor.
The difference between quasi-linear and
quadratic time algorithms 
becomes apparent as $n$ grows.
Note that utilizing Algorithm~\ref{algo:MaxConL} produced
similar error and computational time.
\newpage

\bibliographystyle{abbrv}
\bibliography{ref}

\end{document}